\documentclass[12pt]{amsart}
\usepackage{amsmath}
\usepackage{amssymb}
\usepackage{amsfonts}

\newtheorem{theorem}{Theorem}
\theoremstyle{plain}

\newtheorem{corollary}{Corollary}

\newtheorem{definition}{Definition}
\newtheorem{example}{Example}

\newtheorem{lemma}{Lemma}

\newtheorem{remark}{Remark}

\numberwithin{equation}{section}
\input{tcilatex}

\begin{document}
\title{An extension of the normed dual functors}
\author{Nikica Ugle\v{s}i\'{c}}
\address{Sv. Ante 9, 23287 Veli R\aa t, Hrvatska (Croatia)}
\email{nuglesic@unizd.hr.}
\date{October 26, 2018}
\thanks{This paper is in final form and no version of it will be submitted
for publication elsewhere.}

\begin{abstract}
By means of the direct limit technique, with every normed space $X$ it is
associated a\emph{\ bidualic} (Banach) space $\tilde{X}$ ($D^{2}(\tilde{X}%
)\cong \tilde{X}$ - called the \emph{hyperdual} of $X$) that contains
(isometrically embedded) $X$ as well as all the even (normed) duals $%
D^{2n}(X)$, which make an increasing sequence of the category retracts. The
algebraic dimension $\dim \tilde{X}=\dim X$ ($\dim \tilde{X}=2^{\aleph _{0}}$%
), whenever $\dim X\neq \aleph _{0}$, ($\dim X=\aleph _{0}$). Furthermore,
the correspondence $X\mapsto \tilde{X}$ extends to a faithful covariant
functor (called the \emph{hyperdual functor}) on the category of normed
spaces.
\end{abstract}

\maketitle

\section{Introduction}

In several recent papers (the last two are [15, 16]) the author was solving
the problem of the quotient shape classification of normed vectorial spaces
(especially, the finite quotient shape type classification), which was
initiated by a basic consideration in [14]. Since in a quotient shape theory
the main role play the infinite cardinal numbers, the usual bipolar
separation \textquotedblleft finite-dimensional versus
infinite-dimensional\textquotedblright\ of normed spaces is quite
unsatisfactory. Namely, the class of all infinite-dimensional normed spaces
had to be refined according to General Continuum Hypothesis ($GCH$), and it
had become obvious that the special bases (topological, Schauder, \ldots )
cannot help in solving the problem. The only way has led trough the strict
division by the cardinalities of algebraic (Hamel) bases. This was further
leading to the\emph{\ normed} dual spaces and their algebraic dimensions.
Surprisingly, the author discovered that the inequality $\dim X\leq \dim
X^{\ast }$ was not refined in general. Since this subproblem severely
limited the study of the main one, the author focused his attention to its
solution. In [16], Theorem 4 (by using the shape theory technique), the
answer is given: $\dim X^{\ast }=\dim X$, whenever $\dim X\neq \aleph _{0}$,
while $\dim X=\aleph _{0}$ implies $\dim X^{\ast }=2^{\aleph _{0}}$.
Consequently, every normed dual of every Banach space retains the algebraic
dimension of the space. When, in addition, it became clear that every
canonical embedding of a dual space into its second dual spaces is a
categorical section ([16], Lemma 1 (i) and Theorem 1), the idea of a
consistent embedding of all iterated even (odd) duals into the same Banach
(\textquotedblleft hyperdual\textquotedblright ) space came by itself.

In the realization of the mentioned idea, a property rather close to the
reflexivity (as much as possible) is desired and expected. According to the
result and the example of [8], the first candidates was the somewhat
reflexivity [1, 2]. However, that property (though rather suitable and
useful for a local analysis) is little inappropriate for a global
categorical consideration. Thus (keeping in mind the example of [8]), we had
desired to get an isometric isomorphism between the associated space and its
second dual space. That property is called the \emph{parareflexivity}. By
dropping \textquotedblleft isometric\textquotedblright , the notion of a 
\emph{bidualic} (originally, \emph{bidual-like}) normed space was introduced
in [15], and it also has seemed to be an acceptable one for our final goal.
By adding the somewhat reflexivity to parareflexivity, the obtained notion
of \emph{almost reflexivity} is also considered.

By this work we have succeeded (Theorem 2) to associate with every normed
space $X$ a bidualic (Banach) space $\tilde{X}$, i.e., $D^{2}(\tilde{X}%
)\cong \tilde{X}$, called a \emph{hypercdual} space of $X$, such that $%
\tilde{X}$ contains (canonically embedded) $X$ and all the iterated even
duals $D^{2n}(X)$. Moreover, those duals make a consistently increasing
sequence of category retracts of $\tilde{X}$ having the universal property
(of a direct limit) with respect to the normed spaces and morphisms of norm $%
\leq 1$. Furthermore, the algebraic dimension $\dim \tilde{X}=\dim X$ ($\dim 
\tilde{X}=2^{\aleph _{0}}$), whenever $\dim X\neq \aleph _{0}$ ($\dim
X=\aleph _{0}$). Further (Theorem 3), the correspondence $X\mapsto \tilde{X}$
extends to a faithful covariant functor (called the \emph{hyperdual functor}%
) on the category of normed spaces such that, for every $k\in \{0\}\cup 
\mathbb{N}$, $\tilde{D}D^{2k}=\tilde{D}$ and, for every $X$, $D^{2k}\tilde{D}%
(X)\cong \tilde{D}(X)$. Furthermore, $\tilde{D}$ preserves the
parareflexivity, quasi-reflexivity and reflexivity.

The main working technique is based on the direct limits of the direct
sequences in $i\mathcal{N}_{F}$ (isometries of normed spaces) and the
corresponding in-morphisms between such sequences that admit representatives
having the terms of norm $\leq 1$.

\section{Preliminaries}

We shall implicitly use and apply in the sequel many general and some
special well known facts without referring to any source. Therefore, we
remind a reader that

\noindent - the needed set theoretic and topological facts can be found in
[5];

\noindent - the fundamental facts concerning vectorial, normed and Banach
spaces are learned from [9], [10] and [12];

\noindent\ - the \textquotedblleft categorical Banach space
theory\textquotedblright\ is that of [3] and [6];

\noindent - our category theory terminology strictly follows [7].

\noindent Nevertheless, at least for technical reasons, we think that the
very basic of the categorical approach to normed and Banach spaces (see also
[3, 6]) should be recalled.

Let $\mathcal{V}_{F}$, denote the category of all vectorial spaces over a
field $F$ and all the corresponding linear function. Let $\mathcal{N}_{F}$
denote the category of all normed vectorial spaces and all the corresponding
continuous linear function, whenever $F\in \{\mathbb{R},\mathbb{C}\}$, and
let $\mathcal{B}_{F}$ be the full subcategory of $\mathcal{N}_{F}$
determined by all Banach (i.e., complete normed) spaces. Let

$D:\mathcal{N}_{F}\rightarrow \mathcal{N}_{F}$

\noindent be the normed dual functor, i.e., the contravariant $Hom_{F}$
functor

$D(X)=X^{\ast }$ - the (normed) dual space of $X$,

$D(f:X\rightarrow Y)\equiv D(f)\equiv f^{\ast }:Y^{\ast }\rightarrow X^{\ast
}$, $D(f)(y^{1})=y^{1}f$.

\noindent Then $D[\mathcal{N}_{F}]\subseteq \mathcal{B}_{F}$ and,
furthermore, for every ordered pair $X,Y\in Ob(\mathcal{N}_{F})$, the
function

$D_{Y}^{X}:\mathcal{N}_{F}(X,Y)\equiv L(X,Y)\rightarrow L(Y^{\ast },X^{\ast
})\equiv \mathcal{N}_{F}(Y^{\ast },X^{\ast })$

\noindent is a linear isometry ($\left\Vert D(f)\right\Vert =\left\Vert
f\right\Vert $), and hence, $D$ is a faithful functor.

Further, there exists a covariant Hom-functor

$Hom_{F}^{2}\equiv D^{2}:\mathcal{N}_{F}\rightarrow \mathcal{N}_{F}$,

$D^{2}(X)=D(D(X))\equiv X^{\ast \ast }$ - the (normed) second dual space of $%
X$,

$D^{2}(f:X\rightarrow Y)\equiv D(D(f))\equiv f^{\ast \ast }:X^{\ast \ast
}\rightarrow Y^{\ast \ast }$,

$D^{2}(f)(x^{2})=x^{2}D(f)$.

\noindent Then, clearly, $D^{2}[\mathcal{N}_{F}]\subseteq \mathcal{B}_{F}$
and, for every ordered pair $X,Y\in Ob(\mathcal{N}_{F})$, the function

$(D^{2})_{Y}^{X}:\mathcal{N}_{F}(X,Y)\equiv L(X,Y)\rightarrow L(X^{\ast \ast
},Y^{\ast \ast })\equiv \mathcal{N}_{F}(X^{\ast \ast },Y^{\ast \ast })$

\noindent is a linear isometry ($\left\Vert D^{2}(f)\right\Vert =\left\Vert
f\right\Vert $), and thus, $D^{2}$ is a faithful functor.

\noindent The most useful fact hereby is the existence of a certain natural
transformation $j:1_{\mathcal{N}_{F}}\rightsquigarrow D^{2}$ of the
functors, where, for every $X\in Ob(\mathcal{N}_{F})$, $j_{X}:X\rightarrow
D^{2}(X)$ is an isometric embedding (the \emph{canonical} embedding defined
by $j_{X}(x)\equiv x_{x}^{2}\in D^{2}(X)$, $x\in X$, such that, for every $%
x^{1}\in D(X)$, $x_{x}^{2}(x^{1})=x^{1}(x)\in F$), and the closure $%
Cl(R(j_{X})\subseteq D^{2}(X)$ is the well known (Banach) completion of $X$.
Namely, if $X,Y\in Ob(\mathcal{N}_{F})$, then

$(\forall f\in \mathcal{N}_{F}(X,Y))$, $j_{Y}f=D^{2}(f)j_{X}$

\noindent holds true. Clearly, if $X$ is a Banach space, then the canonical
embedding $j_{X}$ is closed. Continuing by induction, for every $k\in 
\mathbb{N}$, $k>2$, there exists a faithful $Hom_{F}$-functor $D^{k}$ of $%
\mathcal{N}_{F}$ to $\mathcal{N}_{F}$ such that $D^{k}[\mathcal{N}%
_{F}]\subseteq \mathcal{B}_{F}$, $D^{k}$ is contravariant (covariant)
whenever $k$ is odd (even), and for every ordered pair $X$, $Y$ of normed
spaces, the function $(D^{k})_{Y}^{X}$ is an isometric linear morphism of
the normed space $L(X,Y)$ to the Banach space $L(D^{k}(Y),D^{k}(X))$
whenever $k$ is odd ($L(D^{k}(X),D^{k}(Y))$ whenever $k$ is even). Further,
for every $k\in \{0\}\cup \mathbb{N}$, there exists a natural transformation
of the functors $j^{k}:D^{k}\rightsquigarrow D^{k+2}$, where $j^{0}\equiv
j:1_{\mathcal{N}_{F}}\rightsquigarrow D^{2}$ and, for every $k>0$, $j^{k}$
is determined by the class $\{j_{D^{k}(X)}\mid X\in Ob(\mathcal{N}_{F})\}$.
Consequently, there exist the composite natural transformations $%
j^{2k-2}\cdots j^{0}:1_{\mathcal{N}_{F}}\rightsquigarrow D^{2k}$ as well.

\section{Some special limits of normed spaces}

Denote by $i\mathcal{N}_{F}\subseteq \mathcal{N}_{F}$ ($i\mathcal{B}%
_{F}\subseteq \mathcal{B}_{F}$) the subcategory having $Ob(i\mathcal{N}%
_{F})=Ob(\mathcal{N}_{F})$ ($Ob(i\mathcal{B}_{F})=Ob(\mathcal{B}_{F})$) and
for the morphism class $Mor(i\mathcal{N}_{F})$ ($Mor(i\mathcal{B}_{F})$) all
the isometries of $Mor(\mathcal{N}_{F})$ ($Mor(\mathcal{B}_{F})$). Further,
we shall need the subcategories of $\mathcal{B}_{F}\subseteq \mathcal{N}_{F}$
determined by all the contractive morphisms $f$, i.e., for each $x$, $%
\left\Vert f(x)\right\Vert \leq \left\Vert x\right\Vert $, as well as those
determined by all the morphisms having norm $\left\Vert f\right\Vert \leq 1$%
. These are denoted by the subscript $1$ and superscript $1$ respectively.
Clearly, $i\mathcal{N}_{F}\subseteq (\mathcal{N}_{F})_{1}\subseteq (\mathcal{%
N}_{F})^{1}$ and $i\mathcal{B}_{F}\subseteq (\mathcal{B}_{F})_{1}\subseteq (%
\mathcal{B}_{F})^{1}$. We shall also need the sequential in-categories $(%
\func{seq}$-$i\mathcal{N}_{F})_{1}\subseteq (\func{seq}$-$i\mathcal{N}%
_{F})^{1}$ (subcategories of $in$-$\mathcal{N}_{F}=(dir$-$\mathcal{N}%
_{F})/\simeq $) of all direct sequences in $i\mathcal{N}_{F}$ and all the
corresponding (in-)morphisms $\boldsymbol{f}$ admitting representatives $%
(\phi ,f_{n})$ such that all $f_{n}$ belong to $(\mathcal{N}%
_{F})_{1}\subseteq (\mathcal{N}_{F})^{1}$, respectively, and similarly, the
sequential in-categories $(\func{seq}$-$i\mathcal{B}_{F})_{1}\subseteq (%
\func{seq}$-$i\mathcal{B}_{F})^{1}$.

Further, given a functor $F:\mathcal{C}\rightarrow \mathcal{D}$, the $F$%
-image $F[\mathcal{C}]$ is a subcategory of $\mathcal{D}$. We shall need in
the sequel the $D^{2k}$-image and $D^{2k+1}$-image, $k\in \{0\}\cup \mathbb{N%
}$, of the mentioned (sub)categories.

Recall that by the main result of [13] (see also [3], Section 4. (b),
Theorem 4. 1), in the subcategory $(\mathcal{B}_{F})_{1}\subseteq \mathcal{B}%
_{F}$ there exist direct and inverse limits of the corresponding systems.
However, we need a more special and somewhat more general results.

\begin{lemma}
\label{L1}There exist the direct limit functors

$\underrightarrow{\lim }:(\func{seq}$-$i\mathcal{N}_{F})_{1}\rightarrow (%
\mathcal{N}_{F})_{1}$ \quad and

$\underrightarrow{\lim }:(\func{seq}$-$i\mathcal{N}_{F})^{1}\rightarrow (%
\mathcal{N}_{F})^{1}$

\noindent such that, for every $\boldsymbol{X}=(X_{n},i_{nn^{\prime }},%
\mathbb{N})\in Ob(\func{seq}$-$i\mathcal{N}_{F})$,

$\underrightarrow{\lim }\boldsymbol{X}\equiv (X,i_{n})$

\noindent belongs to $i\mathcal{N}_{F}$, and, for every $\boldsymbol{f}\in (%
\func{seq}$-$i\mathcal{N}_{F})^{1}(\boldsymbol{X},\boldsymbol{X}^{\prime })$%
, $\underrightarrow{\lim }\boldsymbol{f}$ belongs to $(\mathcal{N}_{F})_{1}$%
. Furthermore, for every $\boldsymbol{f}\in (\func{seq}$-$i\mathcal{N}_{F})(%
\boldsymbol{X},\boldsymbol{X}^{\prime })$, if $\boldsymbol{f}$ admits a
representative $(\phi ,f_{n})$ such that, for every $n\in \mathbb{N}$, $f_{n}
$ is an isometry (isometric isomorphism), then

$\underrightarrow{\lim }\boldsymbol{f}:\underrightarrow{\lim }\boldsymbol{X}%
\rightarrow \underrightarrow{\lim }\boldsymbol{X}^{\prime }$

\noindent is an isometry (isometric isomorphism).
\end{lemma}

\begin{proof}
Let a direct sequence $\boldsymbol{X}=(X_{n},i_{nn^{\prime }},\mathbb{N})$
in $i\mathcal{N}_{F}$ be given, Consider the disjoint union $\sqcup _{n\in 
\mathbb{N}}X_{n}$ and the binary relation on it defined by

$x_{n}\sim x_{n^{\prime }}^{\prime }\Leftrightarrow (i_{nn^{\prime
}}(x_{n})=x_{n^{\prime }}^{\prime }$ $\vee $ $i_{n^{\prime }n}(x_{n^{\prime
}}^{\prime })=x_{n})$,

\noindent where $x_{n}\in X_{n}$ and $x_{n^{\prime }}^{\prime }\in
X_{n^{\prime }}$. One readily verifies that $\sim $ is an equivalence
relation on $\sqcup _{n\in \mathbb{N}}X_{n}$. Let

$X=(\sqcup _{n\in \mathbb{N}}X_{n})/\sim $

\noindent be the corresponding quotient set. Since all $i_{nn^{\prime }}$
are monomorphisms, for every $x=[x_{n}]\in X$, there exist a unique
(minimal) $n(x)\in \mathbb{N}$ and a unique $x_{n(x)}\in X_{n(x)}$ (the 
\emph{grain} of $x$) such that

$x=[x_{n(x)}]$ $\wedge $ $x_{n(x)}\notin
i_{n(x)-1,n(x)}[X_{n(x)-1}]\trianglelefteq X_{n(x)}$, ($X_{0}\equiv
\varnothing $).

\noindent Furthermore, for every $x\in X$ and every $n\geq n(x)$, there is a
unique $x_{n}=i_{n(x)n}(x_{n(x)}\in X_{n}$ such that $[x_{n}]=[x_{n(x)}]=x$.
And conversely, for every $n$ and every $x_{n}\in X_{n}$, there is a unique $%
x=[x_{n}]\in X$ having the grain $x_{n(x)}\in X_{n(x)}$, $x_{n(x)}\sim x_{n}$
and $n(x)\leq n$. Consequently, every element $x\in X$ is a unique sequence $%
(i_{n(x)n^{\prime }}(x_{n(x)}))_{n^{\prime }\geq n(x)}$, which may be
identified with the vector $x_{n(x)}\in X_{n(x)}\setminus R(i_{n(x)-1,n(x)})$
as well as with the vector $x_{n}\in X_{n}\setminus R(i_{n-1,n}$, $n\geq
n(x) $, and vice versa. Given any $x^{\prime }=[x_{n^{\prime }}^{\prime
}],x^{\prime \prime }=[x_{n^{\prime \prime }}^{\prime \prime }]\in X$, let
us consider

$x_{n_{2}}=i_{n_{1}n_{2}}(y_{n_{1}}^{\prime })+y_{n_{2}}^{\prime \prime }\in
X_{n_{2}}$

\noindent where $n_{1}=\min \{n^{\prime },n^{\prime \prime }\}$, $n_{2}=\max
\{n^{\prime },n^{\prime \prime }\}$, $\{y_{n_{1}}^{\prime
},y_{n_{2}}^{\prime \prime }\}=\{x_{n^{\prime }}^{\prime },x_{n^{\prime
\prime }}^{\prime \prime }\}$ and \textquotedblleft +\textquotedblright\ on
the right side is the addition in $X_{n_{2}}$. Then, for every $n\geq n_{2}$,

$x_{n}=x$'$_{n}+x_{n}^{\prime \prime }\sim i_{n_{1}n_{2}}(y_{n_{1}}^{\prime
})+y_{n_{2}}^{\prime \prime }=x_{n_{2}}$.

\noindent This shows that one can well define

$+:X\times X\rightarrow X$, $+(x^{\prime },x^{\prime \prime })\equiv
x^{\prime }+x^{\prime \prime }=x=[x_{n_{2}}]$,$\in X.$

\noindent (Notice that $n(x)=n_{x^{\prime },x^{\prime \prime }}$ formally
depending on $x^{\prime }$ and $x^{\prime \prime }$, actually depends on the 
$x^{\prime }+x^{\prime \prime }$ only, i.e., it is the unique $n(x^{\prime
}+x^{\prime \prime })$. Namely, if $x^{1}+x^{2}=x^{\prime }+x^{\prime \prime
}=x$, then one readily sees that $n_{x^{1},x^{2}}=n_{x^{\prime },x^{\prime
\prime }}=n(x)\leq \max \{n_{x^{\prime }},n_{x^{\prime \prime }}\}$.) It is
now a routine to verifies that $(X,+)$ is an Abelian group. (For instance,
in order to verify that $(x^{\prime }+x^{\prime \prime })+x^{\prime \prime
\prime }=x^{\prime }+(x^{\prime \prime }+x^{\prime \prime \prime })$,
consider $n=\max \{n_{x^{\prime }},n_{x^{\prime \prime }},n_{x^{\prime
\prime \prime }}\}\geq n_{x^{\prime \prime }+x^{\prime \prime
}},n_{x^{\prime \prime }+x^{\prime \prime \prime })}$.) Further, given an $%
x=[x_{n}]=[x_{n(x)}]\in X$ and a $\lambda \in F$, then $\lambda x_{n}\in
X_{n}$, $\lambda x_{n(x)}\in X_{n(x)}$ and $[\lambda x_{n}]=[\lambda
x_{n(x)}]$. It allows us to define

$\cdot $ $:X\times F\rightarrow X$, $\cdot (x,\lambda )\equiv \lambda
x=[\lambda x_{n(x)}]$.

\noindent One straightforwardly verifies that $X$ with so defined operations
\textquotedblleft $+$\textquotedblright\ and \textquotedblleft $\cdot $%
\textquotedblright\ is a vectorial space over $F$. (Notice that, $n_{\lambda
x}\leq n_{x}$; in order to verify that $\lambda (x^{\prime }+x^{\prime
\prime })=\lambda x^{\prime }+\lambda x^{\prime \prime }$, consider $n=\max
\{n_{x^{\prime }},n_{x^{\prime \prime }}\}\geq n_{x^{\prime \prime
}+x^{\prime \prime }},n_{\lambda x^{\prime }},n_{\lambda x^{\prime \prime }}$%
, while for $\mu (\lambda x)=(\mu \lambda )x$, consider $n=n_{x}\geq
n_{\lambda x},n_{(\mu \lambda )x}$.) Finally, let us define

$\left\Vert \cdot \right\Vert :X\rightarrow \mathbb{R}$, $\left\Vert
x\right\Vert =\left\Vert x_{n(x)}\right\Vert _{n(x)}$,

\noindent where $x=[x_{n(x)}]$ and $x_{n(x)}\in X_{n(x)}$ is the grain of $x$%
. Then, clearly, $\left\Vert x\right\Vert =\left\Vert x_{n(x)}\right\Vert
_{n(x)}=\left\Vert x_{n}\right\Vert $, $n\geq n(x)$. (The function $%
\left\Vert \cdot \right\Vert $ uniquely extends the sequence $(\left\Vert
\cdot \right\Vert _{n})$ of all the norms $\left\Vert \cdot \right\Vert _{n}$
on $X_{n}$ to $X$.) Again, one readily verifies that $\left\Vert \cdot
\right\Vert $ is a well defined norm on $X$. For instance, given any $%
x^{\prime },x^{\prime \prime }\in X$, then (since all $i_{nn^{\prime }}$ are
isometries)

$\left\Vert x^{\prime }+x^{\prime \prime }\right\Vert =\left\Vert
i_{n_{1}n_{2}}(y_{n_{1}}^{\prime })+y_{n_{2}}^{\prime \prime }\right\Vert
_{n_{2}}\leq \left\Vert i_{n_{1}n_{2}}(y_{n_{1}}^{\prime })\right\Vert
_{n_{2}}+\left\Vert y_{n_{2}}^{\prime \prime }\right\Vert _{n_{2}}=$

$=\left\Vert y_{n_{1}}^{\prime }\right\Vert _{n_{1}}+\left\Vert
y_{n_{2}}^{\prime \prime }\right\Vert _{n_{2}}=\left\Vert x_{n(x^{\prime
})}^{\prime }\right\Vert _{n(x^{\prime })}+\left\Vert x_{n(x^{\prime \prime
})}^{\prime \prime }\right\Vert _{n(x^{\prime \prime })}=\left\Vert
x^{\prime }\right\Vert +\left\Vert x^{\prime \prime }\right\Vert $,

\noindent that proves the triangle inequality. Thus, $X\equiv (X,\left\Vert
\cdot \right\Vert )$ is a normed space over $F$. Let us now define, for
every $n\in \mathbb{N}$,

$i_{n}:X_{n}\rightarrow X$, $i_{n}(x_{n})=x=[x_{n}]$.

\noindent Then each $i_{n}$ is linear and, by definition of $\sim $, for
every related pair $n\leq n^{\prime }$, $i_{nn^{\prime }}i_{n^{\prime
}}=i_{n}$ holds. Further, $i_{n}$ is an isometry (and hence, continuous)
because

$\left\Vert i_{n}(x_{n})\right\Vert =\left\Vert x\right\Vert =\left\Vert
x_{n(x)}\right\Vert _{n(x)}=\left\Vert x_{n}\right\Vert _{n}$.

\noindent We have to prove the universal property of $(X,i_{n})$ and $%
\boldsymbol{X}$ with respect to $i\mathcal{N}_{F}$, $(\mathcal{N}_{F})_{1}$
and $(\mathcal{N}_{F})^{1}$. Let, for every $n\in \mathbb{N}$, an isometry $%
f_{n}:X_{n}\rightarrow Y$ (a morphism $f_{n}:X_{n}\rightarrow Y$ of $(%
\mathcal{N}_{F})_{1}$; of $(\mathcal{N}_{F})^{1}$) be given such that

$f_{n^{\prime }}i_{nn^{\prime }}=f_{n}$, $n\leq n^{\prime }$,

\noindent holds. Put

$f:X\rightarrow Y$, $f(x)=f_{n(x)}(x_{n(x)})$,

\noindent where $x_{n(x)}\in X_{n(x)}$ is the grain of $x$. .Then, for every 
$n\geq n(x)$, $f(x)=f_{n}(x_{n})$. Clearly, the function $f$ is well defined
and linear, and, for every $n\in \mathbb{N}$, $fi_{n}=f_{n}$ holds. Further,
for every $x\in X$,

$\left\Vert f(x)\right\Vert _{Y}=\left\Vert f_{n(x)}(x_{n(x)}))\right\Vert
_{Y}=\left\Vert x_{n(x)}\right\Vert _{n(x)}=\left\Vert x\right\Vert $

($\left\Vert f(x)\right\Vert _{Y}=\left\Vert f_{n(x)}(x_{n(x)}))\right\Vert
_{Y}\leq \left\Vert x_{n(x)}\right\Vert _{n(x)}=\left\Vert x\right\Vert $;

$\left\Vert f(x)\right\Vert _{Y}=\left\Vert f_{n(x)}(x_{n(x)}))\right\Vert
_{Y}\leq \left\Vert f_{n(x)}\right\Vert \cdot \left\Vert x_{n(x)}\right\Vert
_{n(x)}\leq \left\Vert x_{n(x)}\right\Vert _{n(x)}=\left\Vert x\right\Vert $)

\noindent implying that $f$ is an isometry (a morphism of $(\mathcal{N}%
_{F})_{1}$, of $(\mathcal{N}_{F})_{1}\subseteq (\mathcal{N}_{F})^{1}$).
Further, assume that $f^{\prime }:X\rightarrow Y$ is any morphism of $%
\mathcal{N}_{F}$ such that, for every $n$, $f^{\prime }i_{n}=f_{n}$. Then,
for every $x\in X$,

$f^{\prime }(x)=f^{\prime }(i_{n(x)}(x_{n(x)}))=f_{n(x)}(x_{n(x)})=f(x)$,

\noindent implying that $f^{\prime }=f$. Therefore, $(X,i_{n})=%
\underrightarrow{\lim }\boldsymbol{X}$ in $i\mathcal{N}_{F}$, in $(\mathcal{N%
}_{F})_{1}$ and in $(\mathcal{N}_{F})^{1}$) (up to isomorphisms of the
category $i\mathcal{N}_{F}$). The constructed direct limit $(X,i_{n}$) of $%
\boldsymbol{X}$ is said to be the \emph{canonical} one. In order to extend
this $\underrightarrow{\lim }$ to a functor, let firstly an

$\boldsymbol{f}=[(\phi ,f_{n})]\in (\func{seq}$-$i\mathcal{N}_{F})_{1}(%
\boldsymbol{X},\boldsymbol{X}^{\prime })$

\noindent be given. We may assume that $(\phi ,f_{n}):\boldsymbol{X}%
\rightarrow \boldsymbol{X}^{\prime }$ is a special representative of $%
\boldsymbol{f}$ (the dual of [11], Lemma I. 1. 2)$,$ i.e., that $\phi $ is
increasing and

$f_{n^{\prime }}i_{nn^{\prime }}=i_{\phi (n)\phi (n^{\prime })}^{\prime
}f_{n}$, $n\leq n^{\prime }$. .

\noindent Let $(X,i_{n})=\underrightarrow{\lim }\boldsymbol{X}$ and $%
(X^{\prime },i_{n}^{\prime })=\underrightarrow{\lim }\boldsymbol{X}^{\prime
} $ be the canonical limits. We define

$f:X\rightarrow X^{\prime }$, $f(x)=i_{\phi (n(x))}^{\prime
}f_{n(x)}(x_{n(x)})$

\noindent (equivalently, $f(x)=i_{\phi (n)}^{\prime }f_{n}(x_{n})$, $%
x=[x_{n}]$).

\noindent Then $f$ is a well defined linear function satisfying

$fi_{n}=i_{\phi (n)}^{\prime }f_{n}$, $n\in \mathbb{N}$.

\noindent Further, since all the $i_{n}^{\prime }$ are isometries, and for
all $n$ and all $x_{n}\in X_{n}$, $\left\Vert f_{n}(x_{n})\right\Vert
_{n}^{\prime }\leq \left\Vert x_{n}\right\Vert _{n}$ holds, it follows that,
for every $x=[x_{n}]\in X$,

$\left\Vert f(x)\right\Vert ^{\prime }=\left\Vert i_{\phi (n)}^{\prime
}f_{n}(x_{n}))\right\Vert ^{\prime }=\left\Vert f_{n}(x_{n}))\right\Vert
_{\phi (n)}^{\prime }\leq \left\Vert x_{n}\right\Vert _{n}=\left\Vert
x\right\Vert $.

\noindent Hence, $f\in (\mathcal{N}_{F})_{1}(X,X^{\prime })$. Let now an $%
\boldsymbol{f}=[(\phi ,f_{n})]\in (\func{seq}$-$i\mathcal{N}_{F})^{1}(%
\boldsymbol{X},\boldsymbol{X}^{\prime })$ be given. By assuming that $(\phi
,f_{n})$ is a special representative as before with $\left\Vert
f_{n}\right\Vert \leq 1$ for all $n$, it follows that

$\left\Vert f(x)\right\Vert ^{\prime }=\left\Vert i_{\phi (n)}^{\prime
}f_{n}(x_{n}))\right\Vert ^{\prime }=\left\Vert f_{n}(x_{n})\right\Vert
_{\phi (n)}^{\prime }\leq \left\Vert f_{n}\right\Vert \cdot \left\Vert
x_{n}\right\Vert _{n}\leq \left\Vert x_{n}\right\Vert _{n}=\left\Vert
x\right\Vert $,

\noindent implying also that $f$ belongs to $(\mathcal{N}_{F})_{1}\subseteq (%
\mathcal{N}_{F})^{1}$. Now, by putting $\underrightarrow{\lim }\boldsymbol{f}%
=f$, one straightforwardly shows that

$\underrightarrow{\lim }:(\func{seq}$-$i\mathcal{N}_{F})_{1}\rightarrow (i%
\mathcal{N}_{F})_{1}$ \quad and

$\underrightarrow{\lim }:(\func{seq}$-$i\mathcal{N}_{F})^{1}\rightarrow (i%
\mathcal{N}_{F})^{1}$

\noindent are functors, i.e., that $\underrightarrow{\lim }\boldsymbol{1}_{%
\boldsymbol{X}}=1_{X}=1_{\underrightarrow{\lim }\boldsymbol{X}}$ and $%
\underrightarrow{\lim }(\boldsymbol{gf})=(\underrightarrow{\lim }\boldsymbol{%
g})(\underrightarrow{\lim }\boldsymbol{f})$ hold true. Since we have already
proven, by the very construction, that $\underrightarrow{\lim }\boldsymbol{X}%
\equiv (X,i_{n})$ belongs to $i\mathcal{N}_{F}$, it remains to verify the
last statement. Let an

$\boldsymbol{f}=[(\phi ,f_{n})]\in (\func{seq}$-$i\mathcal{N}_{F})(%
\boldsymbol{X},\boldsymbol{X}^{\prime }))$

\noindent be given such that all

$f_{n}:X_{n}\rightarrow X_{\phi (n)}^{\prime }$

are isometries. Then, for every $x=[x_{n}]\in X$,

$\left\Vert f(x)\right\Vert ^{\prime }=\left\Vert i_{\phi (n)}^{\prime
}f_{n}(x_{n})\right\Vert ^{\prime }=\left\Vert f_{n}(x_{n})\right\Vert
_{\phi (n)}^{\prime }=\left\Vert x_{n}\right\Vert _{n}=\left\Vert
x\right\Vert $.

\noindent Therefore, $f$ is an isometry. Finally, if all $f_{n}$ are
isometric isomorphisms, then $\boldsymbol{f}$ belongs to $(\func{seq}$-$i%
\mathcal{N}_{F})_{1}$, implying that $f\equiv \underrightarrow{\lim }%
\boldsymbol{f}$ is an isomorphism of $(\mathcal{N}_{F})_{1}$. Therefore, $f$
is an isometric isomorphism, and the proof of the lemma is finished.
\end{proof}

We shall need the following additional facts (related to Lemma 1) in our
forthcoming considerations.

\begin{lemma}
\label{L2}Let $\boldsymbol{X}=(X_{n},i_{nn^{\prime }},\mathbb{N})$ be a
direct sequence in $i\mathcal{N}_{F}$ such that, for every $n$, the bonding
morphism $i_{nn+1}$ is a section of $(\mathcal{N}_{F})_{1}$. Then every
limit morphism $i_{n}:X_{n}\rightarrow \underrightarrow{\lim }\boldsymbol{X}$
is a section of $(\mathcal{N}_{F})_{1}$. $If$, in addition, $X_{n}=X$ for
all n$,$ the $X$ is dominated by $\underrightarrow{\lim }\boldsymbol{X}$ in $%
(\mathcal{N}_{F})_{1}$.
\end{lemma}

\begin{proof}
Given such an $\boldsymbol{X}=(X_{n},i_{nn^{\prime }},\mathbb{N})$ in $i%
\mathcal{N}_{F}$, every $i_{nn+1}$ admits a retractions

$r_{n.n+1}:X_{n+1}\rightarrow X_{n}$

\noindent of $(\mathcal{N}_{F})_{1}$, i.e., $r_{nn+1}i_{nn+1}=1_{X}$ and $%
\left\Vert r_{nn+1}(x_{n+1})\right\Vert \leq \left\Vert x_{n+1}\right\Vert $
(having $\left\Vert r_{n+1}\right\Vert $ $=1$ whenever $X_{n}\neq \{\theta
\} $). Denote, for every related pair $n\leq n^{\prime }$,

$r_{nn^{\prime }}\equiv r_{nn+1}\cdots r_{n^{\prime }-1n^{\prime
}}:X_{n^{\prime }}\rightarrow X_{n}$ \quad ($r_{nn}=1_{X_{n}}$).

\noindent Let $n_{0}\in \mathbb{N}$ be chosen arbitrarily. Put

$r_{n}^{n_{0}}:X_{n}\rightarrow X_{n_{0}}$, $r_{n}^{n_{0}}=\left\{ 
\begin{array}{c}
i_{nn_{0}}\text{, }n\leq n_{0} \\ 
r_{nn_{0}}\text{, }n>n_{0}%
\end{array}%
\right. $.

\noindent Then, for every $n$,

$r_{n+1}^{n_{0}}i_{nn+1}=r_{n}^{n_{0}}$

\noindent holds. By Lemma 1, there exists (in $(\mathcal{N}_{F})_{1}$) an

$r_{n_{0}}:\underrightarrow{\lim }\boldsymbol{X}\rightarrow X_{n_{0}}$, $%
r_{n_{0}}i_{n}=r_{n}^{n_{0}}$, $n\in \mathbb{N}$.

\noindent Then especially, $r_{n_{0}}i_{n_{0}}=i_{n_{0}n_{0}}=1_{X_{n_{0}}}$%
, and the conclusion follows.
\end{proof}

\begin{lemma}
\label{L3}Let $\boldsymbol{X}=(X_{n}^{\prime },i_{nn^{\prime }}^{\prime },%
\mathbb{N})$ and $\boldsymbol{X}^{\prime \prime }=(X_{n}^{\prime \prime
},i_{nn^{\prime }}^{\prime \prime },\mathbb{N})$ be direct sequences in $i%
\mathcal{N}_{F}$ and let $\boldsymbol{f}=[(\phi ,f_{n})]\in (\func{seq}$-$i%
\mathcal{N}_{F})^{1}(\boldsymbol{X}^{\prime },\boldsymbol{X}^{\prime \prime
})$ such that all $f_{n}$ are isomorphisms of $\mathcal{N}_{F}$. Then $%
\underrightarrow{\lim }\boldsymbol{f}$ is an isomorphism of $\mathcal{N}_{F}$
and there exists $\underrightarrow{\lim }(\boldsymbol{f}^{-1})$ such that

$\underrightarrow{\lim }(\boldsymbol{f}^{-1})=($ $\underrightarrow{\lim }%
\boldsymbol{f})^{-1}$.
\end{lemma}

\begin{proof}
We may assume, without loss of generality, that $(\phi ,f_{n})$ is a special
representative of $\boldsymbol{f}$ with $\left\Vert f_{n}\right\Vert \leq 1$
for all $n$. By Lemma 1, there exists

$\underrightarrow{\lim }\boldsymbol{f}\equiv f:\underrightarrow{\lim }%
\boldsymbol{X}^{\prime }\rightarrow \underrightarrow{\lim }\boldsymbol{X}%
^{\prime \prime }$,

\noindent and $f$ belongs to $(\mathcal{N}_{F})_{1}$. Further, since, for
every $n$, $fi_{n}^{\prime }=i_{n}^{\prime \prime }f_{n}$ and $i_{n}^{\prime
}$, $i_{n}^{\prime \prime }$ are the isometries, it readily follows that $%
\left\Vert f_{n}\right\Vert \leq \left\Vert f\right\Vert \leq 1$. We are to
prove that $f$ is an isomorphism of $\mathcal{N}_{F}$. Since all $%
i_{nn^{\prime }}^{\prime }$, $i_{nn^{\prime }}^{\prime \prime }$ and $f_{n}$
are monomorphisms, the construction of the canonical limit implies that $%
\underrightarrow{\lim }\boldsymbol{f}$ is an monomorphism. Let $x^{\prime
\prime }\in X^{\prime \prime }=\underrightarrow{\lim }\boldsymbol{X}^{\prime
\prime }$. Then there exists a unique $x_{n(x^{\prime \prime })}^{\prime
\prime }\in X_{n(x^{\prime \prime })}^{\prime \prime }$ such that $x^{\prime
\prime }=[x_{n(x^{\prime \prime })}^{\prime \prime }]=[x_{n}^{\prime \prime
}]$, $n\geq n(x^{\prime \prime })$. Choose an $n\in \mathbb{N}$ such that $%
\phi (n)\geq n(x^{\prime \prime })$. Since $f_{n}$ is an epimorphism, there
exists an $x_{n}^{\prime }\in X_{n}^{\prime }$ such that $%
f_{n}(x_{n}^{\prime })=x_{\phi (n)}^{\prime \prime }$. Now, there exists a
unique $x^{\prime }=[x_{n}^{\prime }]=i_{n}^{\prime }(x_{n}^{\prime })\in
X^{\prime }$, and it follows, by the very definition of $\underrightarrow{%
\lim }\boldsymbol{f}$, that $f(x^{\prime })=x^{\prime \prime }$. Hence, $f$
is an epimorphism. and consequently, an isomorphism. (If, especially, $%
\left\Vert f_{n}\right\Vert =1$, $n\geq n_{0}$, then $\left\Vert
f\right\Vert =1$.) Let $f^{-1}:X^{\prime \prime }\rightarrow X^{\prime }$ be
the inverse of $f$. Notice that the sequence $(f_{n}^{-1})$ induces the
in-morphism

$\boldsymbol{f}^{-1}=[(\psi ,f_{n}^{-1}i_{n\phi (n)}^{\prime \prime })]:%
\boldsymbol{X}^{\prime \prime }\rightarrow \boldsymbol{X}^{\prime }$.

\noindent Let $f^{-1}:X^{\prime \prime }\rightarrow X^{\prime }$ be the
inverse of $f$. (Caution: In general, $f^{-1}$ does \emph{not} belong to $(%
\mathcal{N}_{F})^{1}\supseteq (\mathcal{N}_{F})_{1}$!) One readily verifies
(by our construction of the direct limit) that, for every $n$,

$f^{-1}i_{\phi (n)}^{\prime \prime }=i_{n}^{\prime }f_{n}^{-1}$

\noindent holds true. Hence, $\underrightarrow{\lim }(\boldsymbol{f}%
^{-1})=f^{-1}$. (Notice that $i_{\phi (n)\phi (n+1)}^{\prime \prime
}f_{n}=f_{n+1}i_{nn+1}^{\prime }$ implies that the sequence $(\left\Vert
f_{n}\right\Vert )$ in $[\left\Vert f_{1}\right\Vert ,1]\subseteq \mathbb{R}$
is increasing and bounded, and, further, that every \textquotedblleft
restriction $f|_{X_{n}^{\prime \prime }}^{X_{n}^{\prime }}$%
\textquotedblright\ carries the norm of $f_{n}$. Therefore, one may say that 
$\left\Vert \underrightarrow{\lim }\boldsymbol{f}\right\Vert =\left\Vert
f\right\Vert =\lim (\left\Vert f_{n}\right\Vert )$!)
\end{proof}

Further, we show that the functors $D^{2k}$ preserve the direct limits of
direct sequences in $i\mathcal{N}_{F}$.

\begin{lemma}
\label{L4}For each $k\in \{0\}\cup \mathbb{N}$, there exist the direct limit
functors

$\underrightarrow{\lim }:(\func{seq}$-$D^{2k}[i\mathcal{N}%
_{F}])_{1}\rightarrow D^{2k}[(\mathcal{N}_{F})_{1}]$ \quad and

$\underrightarrow{\lim }:(\func{seq}$-$D^{2k}[i\mathcal{N}%
_{F}])^{1}\rightarrow D^{2k}[(\mathcal{N}_{F})^{1}]$

\noindent such that, for every $\boldsymbol{X}=(X_{n},i_{nn^{\prime }},%
\mathbb{N})\in Ob(\func{seq}$-$i\mathcal{N}_{F})$,

$\underrightarrow{\lim }D^{2k}[\boldsymbol{X}]\equiv (X^{\prime
},i_{n}^{\prime })\cong (D^{2k}(X),D^{2k}(i_{n}))$

\noindent belongs to $D^{2k}[i\mathcal{N}_{F}]$. Furthermore, for every $%
\boldsymbol{f}\in (\func{seq}$-$D^{2k}[i\mathcal{N}_{F}])(\boldsymbol{X},%
\boldsymbol{X}^{\prime }))$, if $\boldsymbol{f}$ admits a representative $%
(\phi ,f_{n})$ such that, for every $n\in \mathbb{N}$, $f_{n}$ is an
isometry (isometric isomorphism), then

$\underrightarrow{\lim }\boldsymbol{f}:\underrightarrow{\lim }\boldsymbol{X}%
\rightarrow \underrightarrow{\lim }\boldsymbol{X}^{\prime }$

\noindent is an isometry (isometric isomorphism).
\end{lemma}

\begin{proof}
Clearly, every direct sequence in $D^{2k}[i\mathcal{N}_{F}]$ is of the form $%
D^{2k}[\boldsymbol{X}]=(D^{2k}(X_{n}),D^{2k}(i_{nn^{\prime }}),\mathbb{N})$,
where $\boldsymbol{X}=(X_{n},i_{nn^{\prime }},\mathbb{N})$ is a direct
sequence in $i\mathcal{N}_{F}$. Since, by Lemma 1 (i) of [16], all $%
D^{2k}(i_{nn^{\prime }})$ are isometries, every such direct sequence $D^{2k}[%
\boldsymbol{X}]$ belongs to $i\mathcal{N}_{F}$ as well. By Lemma 1, the
direct limit

$\underrightarrow{\lim }D^{2k}[\boldsymbol{X}]=(X^{\prime },i_{n}^{\prime })$%
, $X^{\prime }=((\sqcup _{n\in \mathbb{N}}D^{2k}(X_{n}))/\sim ,\left\Vert
\cdot \right\Vert ^{\prime })$.

\noindent exists in $i\mathcal{N}_{F\text{. }}$ and has the universal
property with respect to $ii\mathcal{N}_{F}$, $(\mathcal{N}_{F})_{1}$ and $(%
\mathcal{N}_{F})^{1}$. We are to prove that $(D^{2k}(X),D^{2k}(i_{n}))$ is a
direct limit of $D^{2k}[\boldsymbol{X}]$ in $D^{2k}[i\mathcal{N}_{F}]$, in $%
D^{2k}[(\mathcal{N}_{F})_{1}]$ and in $D^{2k}[(\mathcal{N}_{F})^{1}]$
(implying that $X^{\prime }$ is isomorphic to $D^{2k}(X)$ in $i\mathcal{N}%
_{F}$, and hence, a Banach space). Firstly, since $j^{2k}:1_{\mathcal{N}%
_{F}}\rightsquigarrow D^{2k}$ is a natural transformation of the functors,
by applying $D^{2k}$ to $\boldsymbol{X}$ and $\underrightarrow{\lim }%
\boldsymbol{X}$, the following commutative diagram

\noindent $%
\begin{array}{cccccc}
X_{1} & \overset{i_{12}}{\rightarrow }\text{ }\cdots & \rightarrow X_{n} & 
\overset{i_{nn+1}}{\rightarrow } & X_{n+1}\hookrightarrow & \cdots \text{ }X
\\ 
\downarrow j_{X_{1}}^{2k} & \cdots & j_{X_{n}}^{2k}\downarrow &  & 
\downarrow j_{X_{n+1}}^{2k}\text{ }\cdots & j_{X}^{2k}\downarrow \\ 
D^{2k}(X_{1}) & \overset{D^{2k}(i_{12})}{\rightarrow }\text{ }\cdots & 
\rightarrow D^{2k}(X_{n}) & \overset{D^{2k}(i_{nn+1})}{\rightarrow } & 
D^{2k}(X_{n+1})\rightarrow & \cdots \text{ }D^{2k}(X)%
\end{array}%
$

\noindent in $i\mathcal{N}_{F}$ occurs, and also

$D^{2k}(i_{n^{\prime }})D^{2}(i_{nn^{\prime }})=D^{2}(i_{n})$, \quad $%
D^{2k}(i_{n})j_{X_{n}}^{2k}=j_{X}^{2k}i_{n}$,

\noindent whenever $n\leq n^{\prime }$. Secondly, we are verifying the
universal property of $(D^{2}(X),D^{2}(i_{n}))$ and $D^{2}[\boldsymbol{X}]$
with respect to the categories $D^{2k}[i\mathcal{N}_{F}$], $D^{2k}[(\mathcal{%
N}_{F})_{1}]$ and $D^{2k}[(\mathcal{N}_{F})^{1}]$. Let, for every $n\in 
\mathbb{N}$, a morphism $D^{2k}(f_{n}):D^{2k}X_{n}\rightarrow D^{2k}(Y)$ of $%
D^{2k}[i\mathcal{N}_{F}]$ (of $D^{2k}[(\mathcal{N}_{F})_{1}]$; of $D^{2k}[(%
\mathcal{N}_{F})^{1}]$) be given such that

$D^{2k}(f_{n^{\prime }})D^{2k}(i_{nn^{\prime }})=D^{2k}(f_{n})$, $n\leq
n^{\prime }$,

\noindent holds. Since each $D^{2k}$ is a faithful functor, it follows that $%
f_{n^{\prime }}i_{nn^{\prime }}=f_{n}$, $n\leq n^{\prime }$. By Lemma 1 (the
case $k=0$), there exists a unique $f:X\rightarrow Y$ of $i\mathcal{N}_{F}$
(of $(\mathcal{N}_{F})_{1}$; of $(\mathcal{N}_{F})^{1}$) such that, for
every $n\in \mathbb{N}$, $fi_{n}=f_{n}$. Then $D^{2k}(f):D^{2k}(X)%
\rightarrow D^{2k}(Y)$ belongs to $D^{2k}[i\mathcal{N}_{F}$] (to $D^{2k}[(%
\mathcal{N}_{F})_{1}]$; to $D^{2k}[(\mathcal{N}_{F})^{1}]$) and, for every $%
n $,

$D^{2k}(f)D^{2k}(i_{n})=D^{2k}f(_{n}$.).

\noindent If $D^{2}(f^{\prime }):D^{2}(X)\rightarrow D^{2}(Y)$ is any
morphism of $D^{2}[i\mathcal{N}_{F}$] (of $D^{2k}[(\mathcal{N}_{F})_{1}]$;
of $D^{2k}[(\mathcal{N}_{F})^{1}]$) such that $D^{2}(f^{\prime
})D^{2}(i_{n})=D^{2}(f_{n}$.)$,$ $n\in \mathbb{N}$, then

$D^{2}(f^{\prime }i_{n})=D^{2}(f_{n})=D^{2}(fi_{n})$ \quad implying

$f^{\prime }i_{n}=f_{n}=fi_{n}$, $n\in \mathbb{N}$.

\noindent Since $f$ is unique having that property, it follows that $%
f^{\prime }=f$, and thus, $D^{2k}(f^{\prime })=D^{2k}(f)$, implying the
uniqueness of $D^{2k}(f)$ in $D^{2k}[i\mathcal{N}_{F}]$ (in $D^{2k}[(%
\mathcal{N}_{F})_{1}]$, in $D^{2k}[(\mathcal{N}_{F})^{1}]$). Therefore, $%
(D^{2k}(X),D^{2k}(i_{n}))$ is a direct limit of $D^{2k}[\boldsymbol{X}]$ in $%
D^{2k}[i\mathcal{N}_{F}]$, in $D^{2k}[(\mathcal{N}_{F})_{1}]$ and in $%
D^{2k}[(\mathcal{N}_{F})^{1}]$. Consequently, by construction of the object
of the canonical direct limit of a direct sequence in $i\mathcal{N}_{F}$, it
follows that $X^{\prime }\cong D^{2k}(X)$ in $D^{2k}[i\mathcal{N}%
_{F}]\subseteq i\mathcal{B}_{F}$, and the statement for objects follows in
general. Concerning the morphisms, let firstly an

$\boldsymbol{f}=[(\phi ,f_{n})]\in (\func{seq}$-$D^{2k}[i\mathcal{N}%
_{F}])_{1}(D^{2}[\boldsymbol{X}],D^{2k}[\boldsymbol{X}^{\prime }])$

\noindent be given. Then we define

$f\equiv \underrightarrow{\lim }\boldsymbol{f}:\underrightarrow{\lim }D^{2k}[%
\boldsymbol{X}]\rightarrow \underrightarrow{\lim }D^{2k}[\boldsymbol{X}%
^{\prime }]$

\noindent as in the proof of Lemma 1, and the functoriality of $%
\underrightarrow{\lim }$ follows straightforwardly. The same holds true for
an

$\boldsymbol{f}=[(\phi ,f_{n})]\in (\func{seq}$-$D^{2k}[i\mathcal{N}%
_{F}])^{1}(D^{2}[\boldsymbol{X}],D^{2k}[\boldsymbol{X}^{\prime }])$.

\noindent Finally, since $D^{2k}$ preserves isometries and isomorphisms, if
every $f_{n}$ is an isometry (isometric isomorphism) then, as in the proof
of Lemma 1, $\underrightarrow{\lim }\boldsymbol{f}$ is an isometry
(isometric isomorphism) as well.
\end{proof}

\begin{theorem}
\label{T1}(i) Each restriction functor

$D^{2k}:i\mathcal{N}_{F}\rightarrow D^{2k}[i\mathcal{N}_{F}]\subseteq i%
\mathcal{B}_{F}$, $k\in \mathbb{N}$,

\noindent preserves directedness of direct sequences and it is continuous,
i.e., it commutes with the direct limit:

$D^{2k}(\underrightarrow{\lim }\boldsymbol{X})\cong \underrightarrow{\lim }%
D^{2k}[\boldsymbol{X}]$ isometrically;

\noindent (ii) Each restriction functor

$D^{2k-1}:i\mathcal{N}_{F}\rightarrow D^{2k-1}[i\mathcal{N}_{F}]\subseteq 
\mathcal{B}_{F}$, $k\in \mathbb{N}$,

\noindent turns direct sequences into inverse sequences and their direct
limits into the corresponding inverse limits, i.e.,

$D^{2k-1}(\underrightarrow{\lim }\boldsymbol{X})\cong \underleftarrow{\lim }%
D^{2k-1}[\boldsymbol{X}]$ (isometrically in $D^{2k-1}[i\mathcal{N}_{F}]$);

\noindent (iii) Each restriction functor

$D^{2k}:D^{2l-1}[i\mathcal{N}_{F}]\rightarrow D^{2k+2l-1}[i\mathcal{N}%
_{F}]\subseteq \mathcal{B}_{F}$, $k,l\in \mathbb{N}$,

\noindent preserves inverseness of inverse sequences and commutes with
inverse limits, i.e.,

$D^{2k}(\underleftarrow{\lim }D^{2l-1}[\boldsymbol{X}])\cong \underleftarrow{%
\lim }D^{2k}[D^{2l-1}[\boldsymbol{X}]]=\underleftarrow{\lim }D^{2k+2l-1}[%
\boldsymbol{X}]$

\noindent (isometrically in $D^{2k+2l-1}[i\mathcal{N}_{F}]$).
\end{theorem}

\begin{proof}
(i). Firstly, by Lemma 1, the needed direct limits exist. Furthermore, by
Lemma 4 and its proof, if $\boldsymbol{X}$ is a direct sequence in $i%
\mathcal{N}_{F}$ and $X\equiv \underrightarrow{\lim }\boldsymbol{X}$ in $i%
\mathcal{N}_{F}$, then, for every $k\in \mathbb{N}$,

$D^{2k}(\underrightarrow{\lim }\boldsymbol{X})\cong D^{2k}(X)\cong 
\underrightarrow{\lim }D^{2k}[\boldsymbol{X}]$

\noindent in $D^{2k}[i\mathcal{N}_{F}]$ holds. Consequently, $D^{2k}(%
\underrightarrow{\lim }\boldsymbol{X})\cong \underrightarrow{\lim }D^{2k}[%
\boldsymbol{X}]$ isometrically.

\noindent (ii). Let $k\in \mathbb{N}$, and let $(X,i_{n})$ be a direct limit
(not necessarily canonical) of a direct sequence $\boldsymbol{X}%
=(X_{n},i_{nn^{\prime }},\mathbb{N})$ in $i\mathcal{N}_{F}$. Then $D^{2k-1}%
\boldsymbol{X}\equiv (D^{2k-1}(X_{n}),D^{2k-1}(i_{nn^{\prime }}),\mathbb{N})$
is an inverse sequence in $D^{2k-1}[i\mathcal{N}_{F}]\subseteq \mathcal{B}%
_{F}$ and there exist morphisms

$D^{2k-1}(i_{n}):D^{2k-1}(X)\rightarrow D^{2k-1}(X_{n})$, $n\in \mathbb{N}$,

\noindent of $D^{2k-1}[i\mathcal{N}_{F}]$ such that

$D^{2k-1}(i_{nn^{\prime }})D^{2k-1}(i_{n^{\prime }})=D^{2k-1}(i_{n})$, $%
n\leq n^{\prime }$.

\noindent We are to verify the universal property of $%
(D^{2k-1}(X),D^{2k-1}(i_{n}))$ and $D^{2k-1}[\boldsymbol{X}]$ with respect
to the category $D^{2k-1}[i\mathcal{N}_{F}]$. Let, for every $n\in \mathbb{N}
$, a morphism $D^{2k-1}(f_{n}):D^{2k-1}(Y)\rightarrow D^{2k-1}(X_{n})$ of $%
D^{2k-1}[i\mathcal{N}_{F}]$ be given such that

$D^{2k-1}(i_{nn^{\prime }})D^{2k-1}(f_{n^{\prime }})=D^{2k-1}(f_{n})$, $%
n\leq n^{\prime }$.

\noindent Then, for every $n\in \mathbb{N}$,

$D^{2k-1}(f_{n^{\prime }}i_{nn^{\prime
}})=D^{2k-1}(f_{n}):D^{2k-1}(Y)\rightarrow D^{2k-1}(X_{n})$,

\noindent and it follows that $f_{n^{\prime }}i_{nn^{\prime }}=f_{n}$,
because the functor $D^{2k-1}$ is faithful. By the universal property of $%
(X,i_{n})$ and $\boldsymbol{X}$ with respect to $i\mathcal{N}_{F}$, there
exists a unique $f:X\rightarrow Y$ of $i\mathcal{N}_{F}$ such that $%
fi_{n}=f_{n}$, $n\in \mathbb{N}$. Then $D^{2k-1}(f):D^{2k-1}(Y)\rightarrow
D^{2k-1}(X)$ belongs to $D^{2k-1}[i\mathcal{N}_{F}$] and

$D^{2k-1}(i_{n})D^{2k-1}(f)=D^{2k-1}(f_{n})$, $n\in \mathbb{N}$.

\noindent Finally, let $D^{2k-1}(f^{\prime }):D^{2k-1}(Y)\rightarrow
D^{2k-1}(X)$ be any morphism of $D^{2k-1}(i\mathcal{N}_{F})$ such that, for
every $n$, $D^{2k-1}(i_{n})D^{2k-1}(f^{\prime })=D^{2k-1}(f_{n})$ holds. Then

$D^{2k-1}(f^{\prime }i_{n})=D^{2k-1}(f_{n})=D^{2k-1}(fi_{n})$ \quad implying

$f^{\prime }i_{n}=f_{n}=fi_{n}$, $n\in \mathbb{N}$.

\noindent Since $f$ is unique having that property, it follows that $%
f^{\prime }=f$, and thus, $D^{2k-1}(f^{\prime })=D^{2k-1}(f)$, implying the
uniqueness of $D^{2k-1}(f)$ in $D^{2k-1}[i\mathcal{N}_{F}]$. Therefore, $%
(D^{2k-1}(X),D^{2k-1}(i_{n}))=\underleftarrow{\lim }D^{2k-1}[\boldsymbol{X}]$
in $D^{2k-1}[i\mathcal{N}_{F}]$ (up to an isomorphism of $D^{2k-1}[i\mathcal{%
N}_{F}]$), and the conclusion follows.

\noindent (iii). Consider the simplest case$,$ i.e., $l=k=1$, i.e., the
restriction functor

$D^{2}:D[i\mathcal{N}_{F}]\rightarrow D^{3}[i\mathcal{N}_{F}]=D^{2}[D[i%
\mathcal{N}_{F}]]$.

\noindent Clearly, every inverse sequence in $D[i\mathcal{N}_{F}$] is of the
form $D[\boldsymbol{X}]=(D(X_{n}),D(i_{nn^{\prime }}),\mathbb{N})$, where $%
\boldsymbol{X}=(X_{n},i_{nn^{\prime }},\mathbb{N})$ is a direct sequence in $%
i\mathcal{N}_{F}$. By (ii), $\underleftarrow{\lim }D[\boldsymbol{X}]\cong D(%
\underrightarrow{\lim }\boldsymbol{X})$ in $D[i\mathcal{N}_{F}]$. Then, by
(i) and (ii),

$D^{2}(\underleftarrow{\lim }D[\boldsymbol{X}])\cong D^{2}(D(%
\underrightarrow{\lim }\boldsymbol{X}))=D^{3}(\underrightarrow{\lim }%
\boldsymbol{X})\cong \underleftarrow{\lim }D^{3}[\boldsymbol{X}]$

\noindent in $D^{3}[i\mathcal{N}_{F}]$. The general case follows in a quite
similar way.
\end{proof}

We shall also need a special case of the following general fact.

\begin{lemma}
\label{L5}Let $i^{\prime }\in i\mathcal{B}_{F}(X^{\prime },Y^{\prime })$ and 
$i^{\prime \prime }\in i\mathcal{B}_{F}(X^{\prime \prime },Y^{\prime \prime
})$ yield the closed direct-sum presentations $Y^{\prime }=R(i^{\prime
})\dotplus Z^{\prime }$ and $Y^{\prime \prime }=R(i^{\prime \prime
})\dotplus Z^{\prime \prime }$, respectively, such that $Z^{\prime }$
continuously linearly embeds into $Z^{\prime \prime }$. Then every $f\in 
\mathcal{B}_{F}(X^{\prime },X^{\prime \prime })$ with $\left\Vert
f\right\Vert <1$ admits an extension $g\in \mathcal{B}_{F}(Y^{\prime
},Y^{\prime \prime })$, $gi^{\prime }=i^{\prime \prime }f$, with $\left\Vert
g\right\Vert <1$. In addition, if $f$ is an isomorphism and $Z^{\prime
}\cong Z^{\prime \prime }$, then there exists an extending isomorphism $g$.
\end{lemma}

\begin{proof}
Since $i^{\prime }$ and $i^{\prime \prime }$ are isometries, the morphism

$u\equiv i^{\prime \prime }f(i^{\prime })^{-1}:R(i^{\prime })\rightarrow
R(i^{\prime \prime })$

\noindent of $\mathcal{B}_{F}$ is well defined, and $\left\Vert u\right\Vert
=\left\Vert f\right\Vert $. By the assumptions on the isometries $i^{\prime
} $ and $i^{\prime \prime }$, each $y^{\prime }\in Y^{\prime }$ ($y^{\prime
\prime }\in Y^{\prime \prime }$) admits a unique presentation

$y^{\prime }=i^{\prime }(x^{\prime })+z^{\prime }$, $x^{\prime }\in
X^{\prime }$, $z^{\prime }\in Z^{\prime }$

($y^{\prime \prime }=i^{\prime \prime }(x^{\prime \prime })+z^{\prime \prime
}$, $x^{\prime \prime }\in X^{\prime \prime }$, $z^{\prime \prime }\in
Z^{\prime \prime }$).

\noindent Since $\left\Vert f\right\Vert <1$ and $Z^{\prime }$ admits a
continuous linear embedding into $Z^{\prime \prime }$, there exists a
contonuous linear embedding

$v:Z^{\prime }\rightarrow Z^{\prime \prime }$, $\left\Vert v\right\Vert
<1-\left\Vert f\right\Vert $.

\noindent Then by

$y=i^{\prime }(x^{\prime })+z^{\prime }\mapsto u(i^{\prime }(x^{\prime
}))+v(z^{\prime })\equiv g(y)$

\noindent a function $g:Y^{\prime }\rightarrow Y^{\prime \prime }$ is well
defined. One readily verifies that $g$ is linear. Since $g=u\dotplus v$, the
Inverse Mapping Theorem (applied to the identity functions on the both
direct-sums and the corresponding direct products with the norm $\left\Vert
\cdot \right\Vert _{1}$) implies that $g$ is continuous, i.e., $g\in 
\mathcal{B}_{F}(Y^{\prime },Y^{\prime \prime })$. The extension property
(commutativity) $gi^{\prime }=i^{\prime \prime }f$ holds obviously. Finally,

$\left\Vert g\right\Vert =\left\Vert u\dotplus v\right\Vert \leq \left\Vert
(u,v)\right\Vert _{1}=\left\Vert u\right\Vert \dotplus \left\Vert
v\right\Vert <\left\Vert f\right\Vert +1-\left\Vert f\right\Vert =1$.

\noindent If, in addition, $f$ is an isomorphism and $Z^{\prime }\cong
Z^{\prime \prime }$, then one can choose $v$ to be an appropriate
isomorphism with $\left\Vert v\right\Vert <1-\left\Vert f\right\Vert $, and
the conclusion follows.
\end{proof}

\section{The hyperdual functor}

Let $X$ be a normed vectorial space over $F\in \{\mathbb{R},\mathbb{C}\}$
and let $k\in \{0\}\cup \mathbb{N}$. By simplifying notations, let

$j_{2k}:D^{2k}(X)\rightarrow D^{2k+2}(X)$

\noindent denote the canonical embedding $j_{D^{2k}(X)}$. Since every $%
j_{2k} $ is an isometry, the direct sequence

$\boldsymbol{\tilde{X}}\equiv \boldsymbol{D}^{2k}(X)=(D^{2k}(X),j_{2k},\{0\}%
\cup \mathbb{N})$, \quad i.e.,

$X\overset{j_{0}}{\rightarrow }D^{2}(X)\overset{j_{2}}{\rightarrow }\cdots 
\overset{j_{2k-2}}{\rightarrow }D^{2k}(X)\overset{j_{2k}}{\rightarrow }%
D^{2k+2}(X)\overset{j_{2k+2}}{\rightarrow }\cdots $,

\noindent in $i\mathcal{N}_{F}$ occurs.

\begin{definition}
\label{D1}Given a normed space $X$, a normed space $\tilde{X}$ is said to be
a \textbf{hyperdual} of $X$ if

\noindent (i) $(\forall k\in \{0\}\cup \mathbb{N})$ there exists an isometry 
$i_{2k}:D^{2k}(X)\rightarrow \tilde{X}$;

\noindent (i) for every normed space $Y$ and every sequence $(f_{2k})$, $%
f_{2k}\in (\mathcal{N}_{F})^{1}(D^{2k}(X),Y)$ satisfying $%
f_{2k+2}j_{D^{2k}(X)}=f_{2k}$, there exists a unique $f\in (\mathcal{N}%
_{F})^{1}(\tilde{X},Y)$ (equivalently, $f\in (\mathcal{N}_{F})_{1}(\tilde{X}%
,Y)$) such that $fj_{2k}=f_{2k}$.
\end{definition}

According to Lemma 1, every normed space has a hyperdual, and moreover, all
hyperduals of an $X$ are mutually isometrically isomorphic.

Recall that a normed space $X$ is said to be \emph{reflexive}, if the
canonical embedding $j_{X}:X\rightarrow D^{2}(X)$ is an epimorphism.$,$
i.e., if $j_{X}$ is an isometric isomorphism (isomorphism of $(\mathcal{N}%
_{F})_{1}$). Then, clearly, $X$ itself must be a Banach space. It is well
known that $X$ is reflexive if and only if $D^{n}(X)$ (for some,
equivalently, every $n$) is reflexive. Obviously, $X$ is reflexive if and
only if, it is isomorphic to a reflexive space. In [15], Lemma 4, the notion
of a \emph{bidual-likeness} was introduced by $D^{2}(X)\cong X$ in $\mathcal{%
N}_{F}$. We shall hereby repeat and strengthen the definition. Before that,
for the sake of completeness, recall briefly (see [1, 2, 4]) that a normed
space $X$ is said to be \emph{somewhat reflexive} (\emph{quasi-reflexive} (%
\emph{of order }$n$)) if, for every infinite-dimensional closed subspace $%
W\trianglelefteq X$, there exists a reflexive infinite-dimensional closed
subspace of $Z\trianglelefteq W$ (if the quotient space $D^{2}(X)/R(j_{X})$
is finite-dimensional ($\dim (D^{2}(X)/R(j_{X}))=n$)). Clearly, the
quasi-reflexivity of order $0$ means reflexivity.

\begin{definition}
\label{D2}A normed space $X$ is said to be \textbf{bidualic} (\textbf{%
parareflexive}), if $X\cong D^{2}(X)$ (isometrically). $X$ is said to be 
\textbf{almost reflexive} if it is parareflexive and somewhat reflexive.
\end{definition}

\begin{example}
\label{E1}All the spaces $l_{p}$ and $L_{p}(n)$, $1<p<\infty $, are
reflexive separable Banach spaces;

\noindent James' space $\mathbb{J}$ of [8] is a non-reflexive almost
reflexive separable Banach space;

\noindent the spaces $l_{1}$ and $c_{0}\trianglelefteq l_{\infty }$ are
non-bidualic separable Banach spaces, while $l_{\infty }$ is a non-bidualic
and non-separable Banach space.
\end{example}

One easily sees that a normed space $X$ is bidualic (parareflexive) if and
only if, it is (isometrically) isomorphic to a bidualic (parareflexive)
space. The following facts are almost obvious.

\begin{lemma}
\label{L6}Let $X$ be a normed space. If

\noindent (i) $X$ is bidualic (parareflexive), then $X$ is a Banach space
and, for every $n\in \mathbb{N}$, $D^{2n}(X)\cong X$ and $D^{2n+1}(X)\cong
D(X)$ (isometrically) and $D^{n}(X)$ is bidualic (parareflexive);

\noindent (ii) $X$ is almost reflexive, then $X$ is a Banach space and, for
every $n\in \mathbb{N}$, $D^{2n}(X)\cong X$ and $D^{2n+1}(X)\cong D(X)$
isometrically, and $D^{2n}(X)$ is almost reflexive.

\noindent (iii) None parareflexive non-reflexive space can be isometrically
embedded into any reflexive space.
\end{lemma}

\begin{proof}
Concerning statement (i), recall that every continuous linear function is
uniformly continuous, and thus, it preserves Cauchy sequences. The rest is
obvious. Concerning staement (ii), one has to verify that $D^{2n}(X)$ is
somewhat reflexive, whenever $X$ is almost reflexive. However, it is an
immediate consequence of $D^{2}(X)\cong X$ isometrically and [16], Lemma 1
(i).

\noindent (iii). Let $X$ be a parareflexive space that is not reflexive
(such is, for instance, James' space $\mathbb{J}$ of [8]). Let $Y$ be any
normed space that admits an isometry $f:X\rightarrow Y$. We have to prove
that $Y$ cannot be reflexive. Assume to the contrary, and consider the
following commutative diagram

$%
\begin{array}{ccc}
X & \overset{f^{\prime }}{\rightarrow } & R(f) \\ 
j_{X}\downarrow &  & \downarrow j_{R(f)} \\ 
D^{2}(X) & \underset{D^{2}(f^{\prime })}{\rightarrow } & D^{2}(R(f))%
\end{array}%
$

\noindent in $\mathcal{N}_{F}$, where $f^{\prime }:X\rightarrow R(f)$ is the
restriction of $f$. By (i), $f^{\prime }$ is an isometric isomorphism of
Banach spaces. Since $R(f)$, being closed in $Y$, is reflexive, it follows
that $j_{R(f)}f$ is an isometric isomorphism. Then $D^{2}(f^{\prime
})j_{X}=j_{R(f)}f$ is an isometric isomorphism as well$,$ implying that so
is $j_{X}$ - a contradiction.
\end{proof}

Let $\rho \mathcal{N}_{F}$, $\alpha \mathcal{N}_{F}$, $\pi \mathcal{N}_{F}$, 
$\sigma \mathcal{N}_{F}$, $\beta \mathcal{N}_{F}$ and $\chi \mathcal{N}_{F}$
($\chi _{n}\mathcal{N}_{F}$) denote the full subcategories of $\mathcal{N}%
_{F}$ (actually, of $\mathcal{B}_{F}$) determined by all the reflexive,
almost reflexive, parareflexive, somewhat reflexive, bidualic and
quasi-reflexiv (of prder $n$) spaces, respectively. Clearly, $\rho \mathcal{N%
}_{F}$ is a full subcategory of all the mentioned subcategories and $\alpha 
\mathcal{N}_{F}\subseteq \pi \mathcal{N}_{F}\subseteq \beta \mathcal{N}_{F}$
holds as well. Further, one readily sees that $\chi _{n}\mathcal{N}%
_{F}\subseteq \chi \mathcal{N}_{F}\subseteq \beta \mathcal{N}_{F}$ also
holds.

\begin{theorem}
\label{T2}For every $X\in Ob(\mathcal{N}_{F})$, every hyperdual $\tilde{X}$
of $X$ has the following properties:

\noindent (i) $\tilde{X}$ is a bidualic Banach space, i.e., $D^{2}(\tilde{X}%
)\cong \tilde{X}$ in $\mathcal{N}_{F}$;

\noindent (ii) all $D^{2k}(X)$, $k\in \mathbb{N}$, embed isometrically into $%
\tilde{X}$ making an increasing sequence of retracts and retracts of $\tilde{%
X}$ in $(\mathcal{B}_{F})_{1}$, implying that, for each $k\in \mathbb{N}$,

$\tilde{X}\cong D^{2k}(X)\dotplus N(r_{2k})$ isometrically;

\noindent (iii) $\dim \tilde{X}=\left\{ 
\begin{array}{c}
\dim X\text{, }\dim X\neq \aleph _{0} \\ 
2^{\aleph _{0}}\text{, }\dim X=\aleph _{0}%
\end{array}%
\right. $.
\end{theorem}

\begin{proof}
Let an $X\in Ob(\mathcal{N}_{F})$ be given. According to Definition 2 and
Lemma 1, it suffices to prove the statements for the canonical direct limit
space $\tilde{X}$ of $\boldsymbol{\tilde{X}}=(D^{2k}(X),j_{D^{2k}(X)}\equiv
j_{2k},\{0\}\cup \mathbb{N})$, i.e.,

$\underrightarrow{\lim }\boldsymbol{\tilde{X}}=(\tilde{X},i_{2k})=((\sqcup
_{k\geq 0}D^{2k}(X))/\sim ),\left\Vert \cdot \right\Vert ),i_{2k})$,

\noindent where $\sim $ is induced by $(j_{2k})$ and the norm $\left\Vert
\cdot \right\Vert $ uniquely extends the sequence $(\left\Vert \cdot
\right\Vert _{2k})$ of norms $\left\Vert \cdot \right\Vert _{2k}$ on $%
D^{2k}(X)$ to $\tilde{X}$, while the limit morphisms into $\tilde{X}$ are
the isometries $i_{2k}:D^{2k}(X)\rightarrow \tilde{X}$. Since $j:1_{\mathcal{%
N}_{F}}\rightsquigarrow D^{2}$ is a natural transformation of the functors,
by applying $D^{2}$ to $\boldsymbol{\tilde{X}}$ and $\underrightarrow{\lim }%
\boldsymbol{\tilde{X}}$, the following commutative diagram

\noindent $%
\begin{array}{cccccc}
X & \overset{j_{0}}{\rightarrow }\text{ }\cdots & \rightarrow D^{2k}(X) & 
\overset{j_{2k}}{\rightarrow } & D^{2k+2}(X)\rightarrow & \cdots \text{ }%
\tilde{X} \\ 
\downarrow j_{0} & \cdots & j_{2k}\downarrow &  & \downarrow j_{2k+2}\text{ }%
\cdots & j_{\tilde{X}}\downarrow \\ 
D^{2}(X) & \overset{D^{2}(j_{9})}{\rightarrow }\text{ }\cdots & \rightarrow
D^{2k+2}(X) & \overset{D^{2}(j_{2k})}{\rightarrow } & D^{2k+4}(X)\rightarrow
& \cdots \text{ }D^{2}(\tilde{X})%
\end{array}%
$

\noindent in $i\mathcal{N}_{F}$ occurs and $D^{2}(i_{2k})j_{2k}=j_{\tilde{X}%
}i_{2k}$. By Lemma 1 and its proof, the canonical direct limit of the direct
sequence $D^{2}[\boldsymbol{\tilde{X}}]\equiv
(D^{2k+2}(X),D^{2}(j_{2k}),\{0\}\cup \mathbb{N})$ is

$\underrightarrow{\lim }D^{2}[\boldsymbol{\tilde{X}}]=(X^{\prime
},i_{2k+2}^{\prime })\equiv ((\sqcup _{k\in \{0\}}D^{2k+2}(X))/\sim ^{\prime
},\left\Vert \cdot \right\Vert ^{\prime }),i_{2k+2}^{\prime })$.

\noindent By Theorem 1 (i), there exists an (isometric) isomorphism

$g:\underrightarrow{\lim }D^{2}[\boldsymbol{\tilde{X}}]=X^{\prime
}\rightarrow D^{2}(\tilde{X})=D^{2}(\underrightarrow{\lim }\boldsymbol{%
\tilde{X}})$.

\noindent We are to prove that $\boldsymbol{\tilde{X}}$ is (in-)isomorphic
to $D^{2}[\boldsymbol{\tilde{X}}]$ in $(\func{seq}$-$i\mathcal{N}_{F})^{1}$.
Since all $j_{2k}$ are the canonical embeddings, Lemma 1 (i) of [16] assures
that all $D^{2}(j_{2k})$ are closed isometric embeddings. Notice that by
excluding (including) $j_{0}$ off $\boldsymbol{\tilde{X}}$ (into $D^{2}[%
\boldsymbol{\tilde{X}}]$) nothing relevant for this consideration is
changing. Let us exclude $j_{0}$ off $\boldsymbol{\tilde{X}}$. By [16],
Corollary 1, for every $k\in \mathbb{N}$, there exist the closed direct-sum
presentations of $D^{2k+2}(X)$, induced by sections $j_{2k}$ and $%
D^{2}(j_{2k-2})$ (having $D(j_{2k-1})$ for a common retraction), with the
same closed complementary subspace. More precisely,

$D^{2k+2}(X)=R(j_{2k})\dotplus N(D(j_{2k-1}))=R(D^{2}(j_{2k-2})\dotplus
N(D(j_{2k-1}))$.

\noindent Therefore, by starting with an isomorphism

$f_{2}:D^{2}(X)\rightarrow D^{2}(X)$, $\left\Vert f_{2}\right\Vert <1$,

\noindent (for instance, $f_{2}=\lambda 1_{D^{2}(X)}$, $0<\lambda <1$), we
may apply Lemma 5 ($X^{\prime }=X^{\prime \prime }=D^{2}(X)$, $Y^{\prime
}=Y^{\prime \prime }=D^{4}(X)$, $i^{\prime }=j_{2}$, $i^{\prime \prime
}=D^{2}(j_{0})$, and $Z^{\prime }=Z^{\prime \prime }=N(D(j_{1}))$ and obtain
an isomorphism

$f_{4}:D^{4}(X)\rightarrow D^{4}(X)$, $\left\Vert f_{4}\right\Vert <1$,

\noindent which extends $f_{2}$. i.e.,

$f_{4}j_{2}=D^{2}(j_{0})f_{2}$.

\noindent Continuing by induction, we obtain a sequence $(f_{2k})$ of
isomorphisms

$f_{2k}:D^{2k}(X)\rightarrow D^{2k}(X)$, $\left\Vert f_{2k}\right\Vert <1$,

\noindent such that

$f_{2k+2}j_{2k}=D^{2}(j_{2k-2})f_{2k}$

\noindent (commutating with the bonding morphisms of $\boldsymbol{\tilde{X}}$
and $D^{2}[\boldsymbol{\tilde{X}}]$). Then the sequence $(f_{2k})$
determines an in-(iso)morphism

$\boldsymbol{f}=[(1_{\mathbb{N}},f_{n})]\in (\func{seq}$-$i\mathcal{N}%
_{F})^{1}(\boldsymbol{\tilde{X}},D^{2}[\boldsymbol{\tilde{X}}])$, $2k\mapsto
k\equiv n$,

\noindent having all $f_{n}$ to be isomorphisms with $\left\Vert
f_{n}\right\Vert <1$. Now, by Lemma 3, the existing limit morphism

$f\equiv \underrightarrow{\lim }\boldsymbol{f}:\underrightarrow{\lim }%
\boldsymbol{\tilde{X}}=\tilde{X}\rightarrow X^{\prime }=\underrightarrow{%
\lim }D^{2}[\boldsymbol{\tilde{X}}]$.

\noindent is an isomorphism. Consequently, the composite $gf$ is an
isomorphism of $\tilde{X}$ onto $D^{2}(\tilde{X})$ (which, in general, is 
\emph{not} the limit morphism $\underrightarrow{\lim }(j_{2k})$!), and
property (i) follows by Lemma 6 (i).

\noindent (ii). By [16], Theorem 1, for every $k\in \mathbb{N}$, the
canonical embedding $j_{2k}$ is a section of $(\mathcal{B}_{F})_{1}$ having
for an appropriate retraction

$D(j_{2k-1}):D^{2k+2}(X)\rightarrow D^{2k}(X)$, $%
D(j_{2k-1})j_{2k}=1_{D^{2k}(X)}$.

\noindent Then the conclusion follows by Lemma 2.

\noindent (iii). This property follows by Theorem 5 of [16]. Namely, if $X$
is finite-dimensional then one may choose $\tilde{X}=X$, while if $\dim
X=\infty $, the\emph{\ normed} dual functor $D$ rises the countable
algebraic dimension ($\aleph _{0}$ to $2^{\aleph _{0}}$) only, and (a
quotient of) a \emph{countable} union in $\mathcal{V}_{F}$ cannot rise an
infinite algebraic dimension.
\end{proof}

\begin{remark}
\label{R1}(a) By Theorem 2 and its proof, the constructed bidualic hyperdual 
$\tilde{X}$ of a normed space $X$ is also a bidualic hyperdual of every even
normed dual space D$^{2k}(X)$, $k\in \{0\}\cup \mathbb{N}$, as well.
Further, by applying the same construction to the direct sequence $%
(D(X)^{2k+1},j_{2k+1},\{0\}\cup \mathbb{N})$, one obtains a bidualic
hyperdual $\widetilde{D(X)}$ of every odd normed dual $D^{2k+1}(X)$ of $X$.

\noindent (b) In the proof of Theorem 2 (i), the application of Lemma 5 has
been essential. If it could hold an appropriate analogue of Lemma 5 for the
isometries (in the very special case of the proof of Theorem 2 (i)), then $%
\tilde{X}$ would be a parareflexive space.
\end{remark}

We now want to extend the direct limit construction $X\mapsto \boldsymbol{%
\tilde{X}}\mapsto \underrightarrow{\lim }\boldsymbol{\tilde{X}}\equiv \tilde{%
X}$ in $i\mathcal{N}_{F}$ to a functor on $\mathcal{N}_{F}$, which is
closely related to all $D^{2k}$ functors.

\begin{theorem}
\label{T3}There exists a covariant functor (the normed \emph{hyperdual
functor})

$\tilde{D}:\mathcal{N}_{F}\rightarrow \mathcal{N}_{F}$, $\quad X\mapsto 
\tilde{D}(X)$, $f\mapsto \tilde{D}(f)$,

\noindent such that $\tilde{D}[\mathcal{N}_{F}]\subseteq \beta \mathcal{B}%
_{F}$, and $\tilde{D}$ does not rise the algebraic dimension but $\aleph
_{0} $ (to $2^{\aleph _{0}}$). Moreover,

\noindent (i) $\tilde{D}$ is faithful;

\noindent (ii) $\tilde{D}(f)$ is an isometry if and only if $f$ is an
isometry;;

\noindent (iii) $\tilde{D}$ is continuous, i.e., it commutes with the direc
limit:

$\tilde{D}(\underrightarrow{\lim }\boldsymbol{X})\cong \underrightarrow{\lim 
}\tilde{D}[\boldsymbol{X}]$ isometrically;

\noindent (iv) $(\forall k\in \{0\}\cup \mathbb{N})$ $\tilde{D}D^{2k}=\tilde{%
D}$;

\noindent (v) $(\forall k\in \{0\}\cup \mathbb{N})(\forall X\in Ob(\mathcal{N%
}_{F}))$ $D^{2k}\tilde{D}(X)\cong \tilde{D}(X)$;

\noindent (vi) for each $k\in \{0\}\cup \mathbb{N})$, there exist an
isometric natural transformation $\iota ^{2k}:D^{2k}\rightsquigarrow \tilde{D%
}$ of the functors;

\noindent (vii) $\tilde{D}[\pi \mathcal{N}_{F}]\subseteq \pi \mathcal{N}_{F}$%
, $\tilde{D}[\chi \mathcal{N}_{F}]\subseteq \chi \mathcal{B}_{F}$, $\tilde{D}%
[\chi _{n}\mathcal{N}_{F}]\subseteq \chi _{n}\mathcal{B}_{F}$, $\tilde{D}%
[\rho \mathcal{N}_{F}]\subseteq \rho \mathcal{B}_{F}$.
\end{theorem}

\begin{proof}
According to Theorem 2, $\tilde{D}$ is well defined on the object class $Ob(%
\mathcal{N}_{F})$ by putting $\tilde{D}(X)=\tilde{X}$, where $\tilde{X}$ is
the object of the canonical direct limit $\underrightarrow{\lim }\boldsymbol{%
\tilde{X}}=(\tilde{X},i_{2k})$, and $\boldsymbol{\tilde{X}}%
=(D^{2k}(X),j_{D^{2k}(X)},\{0\}\cup \mathbb{N})$. Let $f\in \mathcal{N}%
_{F}(X,Y)$. Since, for every $k\in \{0\}\cup \mathbb{N}$, $%
j^{2k}:D^{2k}\rightsquigarrow D^{2}(D^{2k})=D^{k+2}$ is a natural
transformation of the covariant functors, the following diagram

\noindent $%
\begin{array}{cccccc}
X & \overset{j_{X}}{\rightarrow }\text{ }\cdots & \hookrightarrow D^{2k}(X)
& \overset{j_{D^{2k}(X)}}{\rightarrow } & D^{2k+2}(X)\rightarrow & \cdots 
\text{ }\tilde{X} \\ 
f\downarrow & \cdots & D^{2k}(f)\downarrow &  & \downarrow D^{2k+2}(f)\text{%
\quad }\cdots &  \\ 
Y & \overset{j_{Y}}{\rightarrow }\text{ }\cdots & \rightarrow D^{2k}(Y) & 
\overset{j_{D^{2k}(Y)}}{\rightarrow } & D^{2k+2}(Y)\rightarrow & \cdots 
\text{ }\tilde{Y}%
\end{array}%
$

\noindent in $\mathcal{N}_{F}$ commutes. Then the equivalence class $%
[(1_{\{0\}\cup \mathbb{N}},D^{2k}(f))]$ of $(1_{\{0\}\cup \mathbb{N}%
},D^{2k}(f))$ is an in-morphism $\boldsymbol{\tilde{f}}:\boldsymbol{\tilde{X}%
}\rightarrow \boldsymbol{\tilde{Y}}$ of the direct sequences in $i\mathcal{N}%
_{F}$. If $f$ belongs to $(\mathcal{N}_{F})^{1}$, then $\boldsymbol{\tilde{f}%
}\in (\func{seq}$-$i\mathcal{N}_{F})^{1}$, and $\underrightarrow{\lim }%
\boldsymbol{\tilde{f}}$ exists by Lemma 1. In general, we have to construct
(in this special case of direct sequences in $i\mathcal{N}_{F}$) a limit
morphism

$\tilde{f}:\tilde{X}=\underrightarrow{\lim }\boldsymbol{\tilde{X}}%
\rightarrow \underrightarrow{\lim }\boldsymbol{\tilde{Y}}=\tilde{Y}$

\noindent explicitly. Given an $\tilde{x}=[x_{2k(\tilde{x})}]\in \tilde{X}$,
put

$\tilde{f}(\tilde{x})=i_{2k(\tilde{x}),Y}D^{2k(\tilde{x})}(f)(x_{2k(\tilde{x}%
)})$,

\noindent where $x_{2k(\tilde{x})}\in D^{2k(\tilde{x})}(X)$ is the grain of $%
\tilde{x}$. Then $\tilde{f}$ is a well defined linear function. Furthermore, 
$\tilde{f}$ is continuous because, for every $\tilde{x}\in \tilde{X}$,

$\left\Vert \tilde{f}(\tilde{x})\right\Vert \leq \left\Vert f\right\Vert
\cdot \left\Vert \tilde{x}\right\Vert $

\noindent holds. Indeed, each $\tilde{x}\in \tilde{X}$ has a unique grain $%
x_{2k(\tilde{x})}\in D^{2k(\tilde{x})}(X)$ (and conversely), and thus (by
definitions of the norms on $\tilde{X}$ and $\tilde{Y}$), it follows (recall
that the elements of the terms are continuous functionals) that

$\left\Vert \tilde{f}(\tilde{x})\right\Vert =\left\Vert i_{2k(\tilde{x}%
),Y}D^{2k(x)}(f)(x_{2k(\tilde{x})})\right\Vert =\left\Vert D^{2k(\tilde{x}%
)}(f)(x_{2k(\tilde{x})})\right\Vert _{2k(\tilde{x}),Y}=$ \smallskip

$=\left\{ 
\begin{array}{c}
\left\Vert f(x_{0})\right\Vert _{Y}\text{, }\quad k(\tilde{x})=0 \\ 
\left\Vert x_{2k(\tilde{x})}D^{2k(\tilde{x})-1}(f)\right\Vert _{2k(\tilde{x}%
),Y}\text{, }k(\tilde{x})\in \mathbb{N}%
\end{array}%
\right. \leq $ \smallskip

$\leq \left\{ 
\begin{array}{c}
\left\Vert f\right\Vert \cdot \left\Vert x_{0}\right\Vert _{X}\text{, }\quad
k(\tilde{x})=0 \\ 
\left\Vert x_{2k(\tilde{x})}\right\Vert _{2k(\tilde{x})}\cdot \left\Vert
D^{2k(\tilde{x})-1}(f)\right\Vert \text{, }k(\tilde{x})\in \mathbb{N}%
\end{array}%
\right. =$ \smallskip

$=\left\{ 
\begin{array}{c}
\left\Vert f\right\Vert \cdot \left\Vert x_{0}\right\Vert _{X}\text{, }\quad
k(\tilde{x})=0 \\ 
\left\Vert x_{2k(\tilde{x})}\right\Vert _{2k(\tilde{x})}\cdot \left\Vert
f\right\Vert \text{, }k(\tilde{x})\in \mathbb{N}%
\end{array}%
\right. =$ \smallskip

$=\left\Vert f\right\Vert \cdot \left\Vert \tilde{x}\right\Vert $.

\noindent Moreover, if $\left\Vert f\right\Vert \leq 1$, then $\left\Vert 
\tilde{f}\right\Vert \leq 1$. Further, since $\tilde{f}%
i_{2k,X}=i_{2k,Y}D^{2k}(f)$ (by the very definition), it follows that $%
\tilde{f}$ is an isometry whenever $f$ is an isometry. We finally define

$\tilde{D}(f)\equiv \tilde{f}:\tilde{X}\equiv \tilde{D}(X)\rightarrow \tilde{%
D}(Y)\equiv \tilde{Y}$.

\noindent Then $\tilde{D}(1_{X})=1_{\tilde{D}(X)}$ obviously holds. Further,
given an $f\in \mathcal{N}_{F}(X,Y)$ and a $g\in \mathcal{N}_{F}(Y,Z)$,
then, since each $D^{2k}$ is a covariant functor, the definition from above
(see also the diagram) implies that

$\tilde{D}(gf)=\widetilde{gf}=\tilde{g}\tilde{f}=\tilde{D}(g)\tilde{D}(f)$.

\noindent Therefore, $\tilde{D}:N_{F}\rightarrow N_{F}$ is a covariant
functor. By Theorem 2 (i), $\tilde{X}$ is a bidualic Banach space, hence, $%
\tilde{D}[\mathcal{N}_{F}]\subseteq \beta \mathcal{B}_{F}$, while Theorem 2
(iii) assures the statement about algebraic dimension. Let us now verify the
additional properties.

\noindent (i). Let $\tilde{D}(f)=\tilde{D}(f^{\prime }):\tilde{D}%
(X)\rightarrow \tilde{D}(Y)$. Assume to the contrary, i.e., that $f\neq
f^{\prime }$. Then there is an $x\in X$ such that $f(x)\neq f^{\prime }(x)$.
Since $i_{1X}$ and $i_{1Y}$ are monomorphism, it follows that

$\tilde{D}(f)i_{1X}(x)=i_{1Y}f(x)\neq i_{1Y}f^{\prime }(x)=\tilde{D}%
(f^{\prime })i_{1X}(x)$,

\noindent implying that $\tilde{D}(f)(\tilde{x})\neq \tilde{D}(f^{\prime })(%
\tilde{x})$ - a contradiction.

\noindent (ii). It suffices to verify the sufficiency. Let $\tilde{D}(f):%
\tilde{D}(X)\rightarrow \tilde{D}(Y)$ be an isometry. Since $i_{1X}$ and $%
i_{1Y}$ are isometries, it follows that, for every $x\in X$,

$\left\Vert f(x)\right\Vert =\left\Vert i_{1Y}(f(x))\right\Vert =\left\Vert 
\tilde{D}(f)(i_{1X}(x)\right\Vert =\left\Vert i_{1X}(x)\right\Vert
=\left\Vert x\right\Vert $.

\noindent (iii). Since, by (i) and (ii), $\tilde{D}$ is faithful and
preserves isometries, we may apply the proof of Lemma 4 (for $D^{2k}$) to $%
\tilde{D}$ as well, and the statement follows.

\noindent (iv). The equality $\tilde{D}D^{2k}=\tilde{D}$, $k\in \{0\}\cup 
\mathbb{N}$, follows by the definition of $\tilde{D}$. Namely, in the
(defining) direct sequence $\boldsymbol{\tilde{X}}$ for $\tilde{D}(X)=\tilde{%
X}$ one may drop any initial part obtaining the same direct limit space. The
same argument keeps valid for an $f\in \mathcal{N}_{F}(X,Y)$, i.e.,

$\tilde{D}(f)=\tilde{D}(D^{2k}(f))\in \mathcal{N}_{F}(\tilde{D}D^{2k}(X),%
\tilde{D}D^{2k}(Y))=\mathcal{N}_{F}(\tilde{D}(X),\tilde{D}(Y))$.

\noindent (v). This property is a consequence of $\tilde{D}[\mathcal{N}%
_{F}]\subseteq \beta \mathcal{N}_{F}$ and [16], Lemma 1 (i), i.e., the
inductive consequence of $D^{2}\tilde{D}(X)\cong \tilde{D}(X)$.

\noindent (vi). Observe that, for given $X,Y\in OB(\mathcal{N}_{F})$ and
each $k\in \{0\}\cup \mathbb{N}$, the relation

$\tilde{D}(f)i_{2k,X}=i_{2k,Y}D^{2k}(f)$

\noindent follows straightforwardly by the construction of $\tilde{X}$ and $%
\tilde{Y}$ and by the definition of $\tilde{f}$. Hence, for each $k\in
\{0\}\cup \mathbb{N}$, the class

$\{i_{2k,X}:D^{2k}(X)\rightarrow \tilde{X}\mid X\in Ob(\mathcal{N}_{F})\}$

\noindent of the corresponding limit morphisms (each one of them is an
isometry) determines an isometric natural transformation $\iota
^{2k}:D^{2k}\rightsquigarrow \tilde{D}$ of the functors.

\noindent (vii) Let $X$ be a parareflexive space, i.e., $X\in Ob(\pi 
\mathcal{N}_{F})$. Then there exists an isometric isomorphism $%
f:X\rightarrow D^{2}(X)$. By applying $D^{2}$ to $\boldsymbol{\tilde{X}}$
and $\underrightarrow{\lim }\boldsymbol{\tilde{X}}=\tilde{D}(X)$, one
readily obtains an in-morphism

$\boldsymbol{f}=[(1_{\mathbb{N}},f_{n})]:\boldsymbol{\tilde{X}}\rightarrow
D^{2}[\boldsymbol{\tilde{X}}]$, $f_{1}=f$, $f_{n}=D^{2}(f)$,

\noindent with all the $f_{n}$ isometric isomorphisms. Then, by Lemma 1,

$\underrightarrow{\lim }\boldsymbol{f}:\tilde{D}(X)\rightarrow D^{2}(\tilde{D%
}(X))$

\noindent is an isometric isomorphism, and thus, $\tilde{D}(X)\in Ob(\pi 
\mathcal{N}_{F})$. Let $X$ be a quasi-reflexive space of order $n\in
\{0\}\cup \mathbb{N}$, i.e., $X\in Ob(\chi _{n}\mathcal{N}_{F})$. Then $%
D^{2}(X)\cong R(j_{X})\dotplus F^{n}$ in $\mathcal{B}_{F}$. Since the
functors $D^{2k}$ are exact ([6], Proposition 6. 5. 20, or [16], Lemma 3]),
the construction of $\tilde{D}(X)$ and property (iv) straightforwardly imply
that

$D^{2}(\tilde{D}(X))\cong R(j_{\tilde{D}(X)})\dotplus F^{n}$.

\noindent Consequently, $\tilde{D}[\chi _{n}\mathcal{N}_{F}]\subseteq \chi
_{n}\mathcal{B}_{F}$, $\tilde{D}[\chi \mathcal{N}_{F}]\subseteq \chi 
\mathcal{B}_{F}$ and $\tilde{D}[\rho \mathcal{N}_{F}]\subseteq \rho \mathcal{%
B}_{F}$ (the case $n=0$). This completes the proof of the theorem.
\end{proof}

Concerning the somewhat reflexivity and, posteriori, the almost reflexivity
of $\tilde{D}(X)$, we have established the following characterizations.

\begin{theorem}
\label{T4}For every normed space $X$, $\tilde{D}(X)$ is somewhat reflexive
if and only if, for every $n\in \mathbb{N}$, $D^{2n}(X)$ is somewhat
reflexive. Consequently, for every parareflexive space $X$, the following
properties are equivalent:

\noindent (i) $\tilde{D}(X)$ is almost reflexive;

\noindent (ii) $\tilde{D}(X)$ is somewhat reflexive;

\noindent (iii) $(\forall n\in \mathbb{N})$ $D^{2n}(X)$ is somewhat
reflexive.
\end{theorem}

\begin{proof}
Let $X$ be a normed space such that $\tilde{D}(X)$ is somewhat reflexive.
Let an $n\in \mathbb{N}$ be given. Notice that $D^{2n}(X)$ is somewhat
reflexive if and only if $R(i_{2n})\trianglelefteq \tilde{D}(X)$ is somewhat
reflexive, where $i_{2n}:D^{2n}(X)\rightarrow \tilde{D}(X)$ is the
(isometric) limit morphism. If $R(i_{2n})$ is finite-dimensional, there is
nothing to prove. If $R(i_{2n})$ is infinite-dimensional, then the
conclusion follows because it is a closed subspace of $\tilde{D}(X)$.
Conversely, let $X$ be a normed space such that, for every $n\in \mathbb{N}$%
, $D^{2n}(X)$ is somewhat reflexive. If $\tilde{D}(X)$ is
finite-dimensional, then there is nothing to prove. Let $\tilde{D}(X)$ be
infinite-dimensional. Since $\tilde{D}(X)$ is a Banach space, it follows
that $\dim \tilde{D}(X)\geq 2^{\aleph _{9}}$ ($CH$ accepted). Let $%
W\trianglelefteq \tilde{D}(X)$ be a closed infinite-dimensional subspace.
Then $W$ is a Banach space and $\dim W\geq 2^{\aleph _{0}}$. Denote

$W_{2n}\equiv W\cap R(i_{2n})\trianglelefteq \tilde{D}(X)$, $n\in \{0\}\cup 
\mathbb{N}$.

\noindent Then every $W_{2n}$ is a closed subspace of $W$, hence, a Banach
space. Observe that there exists an $n_{0}\in \{0\}\cup \mathbb{N}$ such
that $W_{2n_{0}}$ is infinite-dimensional. Indeed, if all $W_{2n}$ were
finite-dimensional, then $W$ would be at most countably infinite-dimensional
- a contradiction. Let $W_{2n}^{\prime }\trianglelefteq D^{2n}(X)$ be the
inverse image of $W_{2n}$ by $i_{2n}$, i.e., $i_{2n}[W_{2n}^{\prime
}]=W_{2n} $. Since all $i_{2n}$ are isometries of Banach spaces, it follows
that every $W_{2n}^{\prime }$ is a closed subspaces of $D^{2n}(X)$ and $\dim
W_{2n}^{\prime }=\dim W_{2n}$. Especially, $W_{2n_{0}}^{\prime }$ is a
closed infinite-dimensional subspace of $D^{2n_{0}}(X)$. Since $%
D^{2n_{0}}(X) $ is somewhat reflexive, there exists a closed
infinite-dimensional subspace $Z_{2n_{0}}\trianglelefteq W_{2n_{0}}^{\prime
} $ that is reflexive. Since $i_{2n_{0}}$ is an isometry, the subspace $%
Z\equiv i_{2n_{0}}[Z_{2n_{0}}]\trianglelefteq W_{2n_{0}}$ is closed,
infinite-dimensional and reflexive. Consequently, $Z$ is closed,
infinite-dimensional and reflexive subspace of $W$, hence, $\tilde{D}(X)$ is
somewhat reflexive. Then the characterizations of the almost reflexivity
follows by $\tilde{D}[\pi \mathcal{N}_{F}]\subseteq \pi \mathcal{N}_{F}$ of
Theorem 3 (vii).
\end{proof}

By Theorem 4, $\tilde{D}(X)$ cannot be extended towards the almost
reflexivity. Observe that Theorem 3 (vii) can be slightly refined as follows.

\begin{corollary}
\label{C1}let $\mathcal{C}\in \{\chi _{n}\mathcal{N}_{F},\chi \mathcal{N}%
_{F},\beta \mathcal{N}_{F},\pi \mathcal{N}_{F},\rho \mathcal{N}_{F}\}$. Then
the restrictions $\tilde{D}:\mathcal{C}\rightarrow \mathcal{C}$, $\tilde{D}:%
\mathcal{C}_{1}\rightarrow \mathcal{C}_{1}$ and $\tilde{D}:i\mathcal{C}%
\rightarrow i\mathcal{C}$ are covariant functors retaining all the
properties of $\tilde{D}:\mathcal{N}_{F}\rightarrow \mathcal{N}_{F}$.
Furthermore, the restriction functor

$\tilde{D}:\rho \mathcal{N}_{F}\rightarrow \rho \mathcal{N}_{F}$

\noindent is naturally isometrically isomorphic to the identity functor

$1_{\rho \mathcal{N}_{F}}:\rho \mathcal{N}_{F}\rightarrow \rho \mathcal{N}%
_{F}$.
\end{corollary}

\begin{proof}
The last statement only asks for a proof. Let $X$ and $Y$ be reflexive
spaces, i.e., $X,Y\in Ob(\rho \mathcal{N}_{F})$. Then the canonical
embeddings $j_{X}$ and $j_{Y}$ are isometric isomorphisms, as well as all
the $j_{D^{2k}(X)}$ and $j_{D^{2k}(Y)}$. Consequently, all the limit
morphisms $i_{2k,X}:D^{2k}(X)\rightarrow \tilde{D}(X)$ and $%
i_{2k,Y}:D^{2k}(Y)\rightarrow \tilde{D}(Y)$ are isometric isomorphisms.
Especially, $i_{0,X}:X\rightarrow \tilde{D}(X)$ and $i_{0,Y}:Y\rightarrow 
\tilde{D}(Y)$ are isometric isomorphisms. Then, for every $f\in \rho 
\mathcal{N}_{F}(X,Y)$, the following diagram

$%
\begin{array}{ccc}
X & \overset{f}{\rightarrow } & Y \\ 
i_{0,X}\downarrow &  & \downarrow i_{0,Y} \\ 
\tilde{D}(X) & \underset{\tilde{D}(f^{\prime })}{\rightarrow } & \tilde{D}%
(Y))%
\end{array}%
$

\noindent in $\rho \mathcal{N}_{F}$ commutes. Therefore, the class $%
\{i_{0,X}\mid X\in Ob(\rho \mathcal{N}_{F})\}$ determines a natural
transformation $\eta :1_{\rho \mathcal{N}_{F}}\rightsquigarrow \tilde{D}$,
that is an isometric isomorphism of the functors.
\end{proof}

At the end, concerning the dual space of a hyperdual, Theorem 3, Theorem 1
(ii) and [16], Theorem 1, imply the following result:

\begin{corollary}
\label{C2}For every normed space $X,$

$(D(\tilde{D}(X),D(i_{2k}))\cong \underleftarrow{\lim }%
(D^{2k+1}(X),D(j_{D^{2k+1}(X)},\{0\}\cup \mathbb{N}))$

\noindent in $(\mathcal{B}_{F})_{1}$, where all $D(j_{D^{2k+1}(X)})$ and all 
$D(i_{2k})$ are the category retractions. \bigskip
\end{corollary}

\begin{center}
\textbf{References\smallskip }
\end{center}

\noindent \lbrack 1] S. F. Bellenot, \textit{The }$J$\textit{-summ of Banach
spaces}, Journal of Functional. Analysis, \textbf{48} (1982), 95-106.

\noindent \lbrack 2] S. F. Bellenot, \textit{Somewhat reflexive Banach spaces%
}, preprint (1982).

\noindent \lbrack 3] J. M. F. Castillo, \textit{The Hitchhicker Guide to
Categorical Banach Space Theory, Part I}., Extracta Mathematicae, Vol. 25, N%
\'{u}m. 2, (2010), 103-140.

\noindent \lbrack 4] P. Civin and B. Yood, \textit{Quasireflexive Banach
spaces}, Proc. Math. Amer. Soc. \textbf{8} (1957), 906-911.

\noindent \lbrack 5] J. Dugundji, \textit{Topology}, Allyn and Bacon, Inc.
Boston, 1978.

\noindent \lbrack 6] [E] J. M. Erdman, \textit{Functional Analysis and
Operator Algebras - An Introduction}$,$ Version October 4, 2015, Portland
State University (licensed PDF)

\noindent \lbrack 7] H. Herlich and G. E. Strecker, \textit{Category Theory,
An Introduction}, Allyn and Bacon Inc., Boston, 1973.

\noindent \lbrack 8] R. C. James, \textit{Bases and reflexivity of Banach
spaces}, Ann. of Math. \textbf{52} (1950), 518-527.

\noindent \lbrack 9] E. Kreyszig, \textit{Introductory Functional Analysis
with Applications}, John Wiley \& Sons, New York, 1989.

\noindent \lbrack 10] S. Kurepa, \textit{Funkcionalna analiza : elementi
teorije operatora}, \v{S}kolska knjiga, Zagreb, 1990.

\noindent \lbrack 11] S. Marde\v{s}i\'{c} and J. Segal, \textit{Shape Theory}%
, North-Holland, Amsterdam, 1982.

\noindent \lbrack 12] W. Rudin, \textit{Functional Analysis, Second Edition}%
, McGraw-Hill, Inc., New York, 1991.

\noindent \lbrack 13] Z. Semadeni and H. Zindengerg, \textit{Inductive and
inverse limits in the category of Banach spaces}, Bull. Acad. Polon. Sci. S%
\'{e}r. Sci. Math. Astronom. Phys. \textbf{13} (1965), 579-583.

\noindent \lbrack 14] N. Ugle\v{s}i\'{c}, \textit{The shapes in a concrete
category}, Glasnik. Mat. Ser. III \textbf{51}(\textbf{71}) (2016), 255-306.

\noindent \lbrack 15] N. Ugle\v{s}i\'{c}, \textit{The quotient shapes of
normed spaces and application}, submitted.

\noindent \lbrack 16] N. Ugle\v{s}i\'{c}, \textit{On the duals of normed
spaces and quotient shapes}, submitted.

\end{document}